\documentclass[oneside, 12pt]{amsart}
\usepackage[utf8]{inputenc}
\usepackage{hyperref}
\usepackage{fancyhdr}
\usepackage{paralist}
\usepackage{multirow}
\usepackage{threeparttable}
\usepackage{footnote}
\makesavenoteenv{tabular}
\usepackage{amsmath}
\usepackage{amsfonts}
\usepackage{amssymb}
\usepackage{amsthm}
\usepackage{bbm}
\usepackage{mathtools}
\usepackage{nicefrac}
\usepackage{setspace}
\usepackage{etoolbox}
\usepackage{tikz-cd}
\tikzset{
commutative diagrams/.cd,
arrow style=tikz,
diagrams={>=latex}
}
\usetikzlibrary{arrows}
\usepackage[arrow, matrix, curve, color]{xy}
\usepackage{marvosym}
\usepackage{pdflscape}
\usepackage{float}
\usepackage{cancel}
\usepackage{soul}
\usepackage{stmaryrd}
\newcommand{\bbC}{\mathbb{C}}
\newcommand{\bbD}{\mathbb{D}}

\newcommand{\bbF}{\mathbb{F}}

\newcommand{\bbI}{\mathbb{I}}
\newcommand{\bbJ}{\mathbb{J}}

\newcommand{\bbN}{\mathbb{N}}
\newcommand{\bbO}{\mathbb{O}}

\newcommand{\bbQ}{\mathbb{Q}}

\newcommand{\bbY}{\mathbb{Y}}
\newcommand{\bbZ}{\mathbb{Z}}
\newcommand{\bfH}{\mathbf{H}}

\newcommand{\bfM}{\mathbf{M}}

\newcommand{\calA}{\mathcal{A}}
\newcommand{\calB}{\mathcal{B}}
\newcommand{\calC}{\mathcal{C}}
\newcommand{\calD}{\mathcal{D}}
\newcommand{\calE}{\mathcal{E}}
\newcommand{\calF}{\mathcal{F}}
\newcommand{\calG}{\mathcal{G}}
\newcommand{\calH}{\mathcal{H}}

\newcommand{\calL}{\mathcal{L}}
\newcommand{\calM}{\mathcal{M}}
\newcommand{\calN}{\mathcal{N}}

\newcommand{\calR}{\mathcal{R}}
\newcommand{\calS}{\mathcal{S}}
\newcommand{\calT}{\mathcal{T}}

\newcommand{\calW}{\mathcal{W}}

\newcommand{\frakm}{\mathfrak{m}}

\newcommand{\bbNe}{\bbN_{\geq 1}}
\newcommand{\fa}{\forall\,}
\newcommand{\esex}{\exists\:}
\newcommand{\inv}{^{-1}}
\newcommand{\eingeschr}[1]{_{\vert #1}}
\newcommand{\id}[1][]{\operatorname{id}_{#1}}
\newcommand{\unterstr}[1]{{\underline{\smash{#1}}}}
\newcommand{\mbo}{\mbox{ }}
\newcommand{\startabc}{\begin{compactenum}[a)]}
\newcommand{\stopabc}{\end{compactenum}}
\newcommand{\op}{^{\operatorname{op}}}
\newcommand{\mor}[2]{#1 \left(#2 \right)}
\newcommand{\katalg}[1]{\unterstr{\operatorname{Alg}}_{#1}}
\newcommand{\moralg}[2]{\mor{\katalg{#1}}{#2}}
\newcommand{\katcalg}[1]{\unterstr{\operatorname{CAlg}}_{#1}}
\newcommand{\dimension}[2][]{\operatorname{dim}_{#1}\left(#2\right)}
\newcommand{\einheitsmatrix}[1][]{\mathbbm{1}_{#1}}
\newcommand{\Kern}[1]{\operatorname{Ker}\left(#1\right)}
\newcommand{\Kokern}[1]{\operatorname{Coker}\left(#1\right)}
\newcommand{\Hom}[2][]{\operatorname{Hom}_{#1}\left( #2\right)}
\newcommand{\End}[2][]{\operatorname{End}_{#1}\left( #2\right)}
\newcommand{\Aut}[2][]{\operatorname{Aut}_{#1}\left( #2\right)}
\newcommand{\Der}[2][K]{\operatorname{Der}_{#1}\left( #2 \right)}
\newcommand{\Ext}[3][1]{\operatorname{Ext}_{#2}^{#1}\left(#3\right)}
\newcommand{\lsc}{\underset{\rightarrow}{.}}
\newcommand{\rsc}{\underset{\leftarrow}{.}}
\newcommand{\knotensub}[1][Q]{{\mathsf{v} \left( #1 \right)}}
\newcommand{\chara}[1]{\operatorname{char}\left(#1\right)}
\newcommand{\stab}[2][]{S(#2)}
\newcommand{\orbit}[1]{\bbO_{#1}}
\newcommand{\orbitmap}[1]{\vartheta_{#1}}
\newcommand{\spec}[1]{\operatorname{Spec}\left(#1\right)}
\newcommand{\eintragQuelle}[7][1=2]{
\ifthenelse{#1}{
\bibitem[#2]{#3}
#4, \emph{#5}, #6 #7.
}{}
}
\newtheoremstyle{pkt}
{5pt}
{5pt}
{}
{}
{\bf}
{}
{.5em}
{}
\theoremstyle{pkt}
\newtheorem{de}{Definition}[section]
\newtheorem{lem}[de]{Lemma}
\newtheorem{bsp}[de]{Example}
\newtheorem{sa}[de]{Proposition}
\newtheorem{kor}[de]{Corollary}
\newtheorem{thm}[de]{Theorem}
\makeatletter
\renewenvironment{proof}[1][\proofname] {\par\pushQED{\qed}\normalfont\topsep6\p@\@plus6\p@\relax\trivlist\item[\hskip\labelsep{\itshape#1}\@addpunct{:}]\ignorespaces}{\popQED\endtrivlist\@endpefalse}
\makeatother
\newcommand{\adjunction}[2]{\underset{#2}{\overset{#1}{\longleftrightarrows}}}
\usepackage{stackrel}
\newcommand{\leftrarrows}{\mathrel{\raise.75ex\hbox{\oalign{%
  $\scriptstyle\leftarrow$\cr
  \vrule width0pt height.5ex$\hfil\scriptstyle\relbar$\cr}}}}
\newcommand{\lrightarrows}{\mathrel{\raise.75ex\hbox{\oalign{%
  $\scriptstyle\relbar$\hfil\cr
  $\scriptstyle\vrule width0pt height.5ex\smash\rightarrow$\cr}}}}
\newcommand{\lefthooks}{\mathrel{\raise.75ex\hbox{\oalign{%
  $\scriptstyle\hookrightarrow$\cr
  \vrule width0pt height.5ex$\hfil\scriptstyle\relbar$\cr}}}}
\newcommand{\righthooks}{\mathrel{\raise.75ex\hbox{\oalign{%
  $\scriptstyle\relbar$\hfil\cr
  $\scriptstyle\vrule width0pt height.5ex\smash{\rotatebox{180}{\reflectbox{$\hookrightarrow$}}}$\cr}}}}
\newcommand{\Rrelbar}{\mathrel{\raise.75ex\hbox{\oalign{%
  $\scriptstyle\relbar$\cr
  \vrule width0pt height.5ex$\scriptstyle\relbar$}}}}
\newcommand{\longleftrightarrows}{\leftrarrows\joinrel\Rrelbar\joinrel\lrightarrows}

\makeatletter
\renewcommand\@seccntformat[1]{\csname prepend@#1\endcsname}
\newcommand{\prepend@section}{\thesection\quad}
\newcommand{\prepend@subsection}{\thesubsection\quad}
\newcommand{\prepend@subsubsection}{\thesubsubsection\quad}
\makeatother
\newcommand{\GITquot}[2]{{#1}/\!\!/{#2}}
\newcommand{\cplhallalg}[1]{\bfH(\! ( #1 )\!)}
\newcommand{\dimvec}{\unterstr{\operatorname{dim}}}
\newcommand{\cplmonoidalg}[2][\bbQ]{{#1} [\! [{#2} ]\! ]}
\newcommand{\twmonoidalg}[3][\bbQ]{{#1}^{\text{{#2}tw}} [{#3}]}
\newcommand{\twcplmonoidalg}[3][\bbQ]{{#1}^{\text{{#2}tw}} [\! [{#3} ]\! ]}
\newcommand{\ev}[1][q]{\operatorname{ev}_{#1}}
\newcommand{\GL}[1]{\mathbf{GL}_{#1}}
\newcommand{\PGL}[1]{\mathbf{PGL}_{#1}}
\newcommand{\SL}[1]{\mathbf{SL}_{#1}}
\newcommand{\PSL}[1]{\mathbf{PSL}_{#1}}
\newcommand{\HNN}[3]{{#1}*_{#2}^{#3}}
\newcommand{\doppelpfeil}[2]{\underset{#2}{\overset{#1}{\rightrightarrows}}}
\newcommand{\Exp}{\operatorname{Exp}}
\newcommand{\Log}{\operatorname{Log}}
\newcommand{\mongrad}[1][.]{\vert\,#1\,\vert}
\newcommand{\simp}[1][]{\operatorname{sim}_{#1}}
\newcommand{\ssimp}[1][]{\operatorname{ssim}_{#1}}
\newcommand{\absimp}[1][]{\operatorname{absim}_{#1}}
\newcommand{\iso}[1][]{\operatorname{iso}_{#1}}
\newcommand{\correctionform}{{\bbY}}
\newcommand{\countfct}[3][sim]{r^{\operatorname{#1}#2}_{#3}}
\newcommand{\countpol}[3][sim]{R^{\operatorname{#1}#2}_{#3}}
\newcommand{\countgl}[1][a]{{P_{\GL{#1}}}}
\newcommand{\modulirep}[3][]{M^{\operatorname{#1}}(#2,#3)}
\newcommand{\fakesubsection}[1]{%
  \par\refstepcounter{subsection}
  \subsectionmark{#1}
  \addcontentsline{toc}{subsection}{\protect\numberline{\thesubsection}#1}
}
\newcommand{\fakesubsubsection}[1]{%
  \par\refstepcounter{subsubsection}
  \subsubsectionmark{#1}
  \addcontentsline{toc}{subsubsection}{\protect\numberline{\thesubsubsection}#1}
}
\newcommand{\pr}{\operatorname{pr}}
\newcommand{\localpol}{\bbQ[s]_{(s-q)}}
\newcommand{\thmabsatz}{\mbo\mbo\mbo}
\newcommand{\defit}[2][1]{{\textit{#2}}}
\newcommand{\paramExtDihed}{c}
\begin{document}
\title[Arithmetic representation growth of virtually free groups]
      {Arithmetic representation growth of virtually free groups}
\author[Fabian Korthauer]{Fabian Korthauer}
\address{Mathematisches Institut, Heinrich-Heine-Universit\"at, 40204 D\"usseldorf, Germany}
\curraddr{}
\email{Korthauer.maths@gmx.de}
\subjclass[2010]{20C07, 16G99, 14L30, 14D20, 32S35}
\begin{abstract}
We adapt methods from quiver representation theory and Hall \mbox{algebra} techniques to the counting of representations of virtually free groups over finite fields. This gives rise to the computation of the E-polynomials of \mbox{$\mathbf{GL}_d(\mathbb{C})$-character} varieties of virtually free groups. As examples we discuss the representation theory of $\mathbb{D}_\infty$ , $\mathbf{PSL}_2(\mathbb{Z})$ , $\mathbf{SL}_2(\mathbb{Z})$ , $\mathbf{GL}_2(\mathbb{Z})$\text{ and }$\mathbf{PGL}_2(\mathbb{Z})$ .
\end{abstract}
\maketitle
\section{Introduction}
\noindent Arithmetic representation growth deals with counting the number of \mbox{representations} of algebras over finite fields. More precisely it is the study of the following \defit{counting functions:} For $\calA$ a finite type\footnote{i.e. finitely generated as an $\bbF_q$-algebra} $\bbF_q$-algebra and $d\in\bbN_0, \alpha\in\bbNe$ define 
\begin{equation}
\label{equ_def_countfct_ordinary}
\begin{array}{ccl}
\countfct[ss]{,\calA}{d}(q^\alpha) &:=& \#\ssimp[d]\left(\calA\otimes_{\bbF_q} \bbF_{q^\alpha}\right)
\\[2pt]
\countfct[sim]{,\calA}{d}(q^\alpha) &:=& \#\simp[d]\left(\calA\otimes_{\bbF_q} \bbF_{q^\alpha}\right)
\\[2pt]
\countfct[absim]{,\calA}{d}(q^\alpha) &:=& \#\absimp[d]\left(\calA\otimes_{\bbF_q} \bbF_{q^\alpha}\right)
\end{array}
\end{equation}
Here we denote by $\iso[d](\calB) \supseteq\ssimp[d](\calB)\supseteq\simp[d](\calB)\supseteq\absimp[d](\calB)$ for each $d$ the sets of \mbox{isomorphism} classes of all, of all semisimple, of all simple and of all absolutely simple left modules $\calM$ over a $K$-algebra $\calB$ of dimension $\dimension[K]{\calM}=d$ . Recall that a left $\calB$-module $\calM$ is called \defit{absolutely simple} if it is simple and $\End[\calB]{\calM}=K$ or equivalently if $\calM\otimes_K \overline{K}$ is simple for the algebraic closure $\overline{K}\supseteq K$ .\par%
\eqref{equ_def_countfct_ordinary} defines functions $\countfct[ss]{,\calA}{d}, \countfct[sim]{,\calA}{d}, \countfct[absim]{,\calA}{d}$ on all $q$-powers. We call these \mbox{functions} counting functions, as they count the semisimple, simple and \mbox{absolutely} simple \mbox{modules/representations} of $\calA$ over $\bbF_q$ up to isomorphism. If the algebra $\calA$ is \mbox{understood,} we will usually drop it from the notation.\par%
The counting functions \eqref{equ_def_countfct_ordinary} have been studied by S. Mozgovoy and M. Reineke in the cases $\calA=\bbF_q\vec{Q}$ the path algebra of a finite quiver and $\calA=\bbF_q[F_a]$ the group algebra of a finitely generated free group. One of their main results is the following theorem.\footnote{In the quiver case Mozgovoy-Reineke's result was in the more general context of counting \mbox{absolutely} stable representation of a fixed dimension vector. We state it in a weaker form here for expository purposes.}
\begin{thm}\footnote{see \cite[Thm. 6.2]{xxreinekecounting} and \cite[Thm. 1.1]{xxMozgovoyReineke2015}}\thmabsatz
\label{thm_state_of_the_art}
If $\calA$ is the path algebra of a finite quiver or the group algebra of a finitely generated free group (over $\bbF_q$ respectively), then there are polynomials \mbox{$\countpol[ss]{}{d}, \countpol[absim]{}{d}\in\bbZ[s]$} and $\countpol[sim]{}{d}\in\bbQ[s]$ fulfilling
\begin{equation}
\label{equ_state_of_the_art}
\fa \alpha\geq 1: 
\countpol[absim]{}{d}(q^\alpha) = \countfct[absim]{}{d}(q^\alpha)
\; , \;
\countpol[sim]{}{d}(q^\alpha) = \countfct[sim]{}{d}(q^\alpha)
\; , \;
\countpol[ss]{}{d}(q^\alpha) = \countfct[ss]{}{d}(q^\alpha)
\end{equation}
\end{thm}
We call such polynomials realizing the counting functions \eqref{equ_def_countfct_ordinary} \defit{counting \mbox{polynomials.}} More important than the mere existence of counting polynomials is the fact that Mozgovoy-Reineke obtained certain generating formulas which enable the practical computation of the counting polynomials (in low dimensions).\par%
The main goal of this paper is to generalize Theorem \ref{thm_state_of_the_art} as well as the above mentioned generating formulas to the case where $\calA = \bbF_q[\calG]$ is the group algebra of a finitely generated virtually free group (see Theorem \ref{MAIN THEOREM} below). Furthermore we will investigate a few structural properties of the counting polynomials and relate these to the geometry of GIT moduli spaces of representations. SageMath code designed by the author for the practical computation of the counting polynomials is provided as an ancillary file to this publication on the arXiv.\par%
This paper is organized as follows: We start by recalling most of the relevant \mbox{preknowledge} on virtually free groups and algebraic geometry within Section \ref{section_preliminaries}. \mbox{Afterwards} we discuss some invariants like dimension vectors and the \mbox{homological} Euler form in the context of representations of virtually free groups in Section \ref{section_3_some_invariants_of_virtually_free_groups}. In Section \ref{section_4_hall_algebra_methods} we review some Hall algebra methods, before we discuss the main result \ref{MAIN THEOREM} in Section \ref{section_5_counting_polynomials}. Section \ref{section_examples} is devoted to hands-on examples and we conclude in Section \ref{section_structural_properties} with discussing a few structural properties of the counting polynomials.\par%
\subsection*{Acknowledgements}
Almost all results of this paper have been obtained as part of the author's PhD under the supervision of M. Reineke and B. Klopsch and the author is deeply indebted to his supervisors for all their help, support, kindness and patience. Special thanks also go to L. Hennecart, S. Mozgovoy and S. Schröer for useful discussions on parts of the content of this paper. Moreover the author thanks the Ruhr-University Bochum and most of its members for being a welcoming scientific home during his PhD. During his PhD the author's research was financed by the research training group \emph{GRK 2240: Algebro-Geometric Methods in Algebra, Arithmetic and Topology}, which is funded by the Deutsche Forschungsgemeinschaft.
\section{Preliminaries}
\label{section_preliminaries}
\noindent In this section we summarize the preknowledge from group theory, representation theory and algebraic geometry needed within this paper. Except for the notion of suitable fields nothing in this section is original.
\subsection{Virtually free groups and their group algebras}
\label{subsec_virtually_free_groups_and_their_group_algebras}
We start by recalling some group theoretic notions. Given three groups $\calE,\calF,\calH$ as well as two injective group homomorphisms $\iota:\calF \hookrightarrow \calH$ , $\kappa: \calF\hookrightarrow \calE$ we denote their pushout in the category of groups by $\calH *_\calF \calE$ and call it the \defit{amalgamated free product} of $\calH$ and $\calE$ over $\calF$ . If we are given two embeddings $\iota,\kappa:\calF\hookrightarrow \calH$ with the same codomain, we consider the induced embeddings
\[
\iota',\kappa' :\calF \hookrightarrow \calH * C_\infty \quad, \quad \iota'(f):= \iota(f) \quad ,\quad  \kappa'(f):=t\inv\kappa(f)t
\]
where $t$ denotes the generator of the infinite cyclic group $C_\infty$ . We denote the coequalizer of
\[
\begin{tikzcd}
\calF \arrow[r,hook,shift left=0.5ex, "{\iota'}"] \arrow[r,hook',shift right=0.5ex,"\kappa'"'] & \calH * C_\infty
\end{tikzcd}
\]
by $\HNN{\calH}{\calF}{\iota,\kappa}$ and call it the \defit{HNN extension} of $\calH$ by $\calF$ . Even though our description of amalgamated free products and HNN extensions make sense if the homomorphisms $\iota$ and $\kappa$ are not injective, we will follow the usual convention in group theory and only consider the case where they are.\par%
A group $\calG$ is called \defit{virtually free} if it contains a finite index subgroup $\calG$ which itself is a free group. However, Bass-Serre theory provides an equivalent definition of virtually free groups which will be more useful for us:\footnote{see \cite[§II.2.6]{xxtrees}, \cite[Thm. 1]{xxkarrass73}} A finitely generated group $\calG$ is virtually free if and only if it is of the form
\begin{equation}
\label{equ_virt_free_decomp}
\calG \cong \left( \ldots \left(\left( \ldots \left(\left(\calG_0 *_{\calG_1'} \calG_1 \right)*_{\calG_2'} \calG_2\right) \ldots \right) \HNN{}{\calG_{I+1}'}{\iota_{I+1},\kappa_{I+1}} \right) \ldots \right) \HNN{}{\calG_{I+J}'}{\iota_{I+J},\kappa_{I+J}}
\end{equation}
for finite groups $\calG_0,\ldots ,\calG_I$ and $\calG_{1}' ,\ldots ,\calG_{I+J}'$ . Note that $I$ denotes the number of \mbox{amalgamated} free products and $J$ the number of HNN extensions in \eqref{equ_virt_free_decomp}. Each of the finite groups $\calG_i$ and $\calG_j'$ is embedded into $\calG$ as a subgroup in a canonical way. We may always assume that there are maps \mbox{$s,t:\{1,\ldots ,I+J\}\to \{0,\ldots ,I\}$} fulfilling
\begin{equation}
\label{equ_virt_free_maps1_s_t}
t(j) = 
\begin{cases}
j &, \quad \text{if } j\leq I
\\
s(j) &. \quad \text{if } j>I
\end{cases}
\quad ,\quad
s(j) \in\{0,\ldots,j-1\} \quad \text{if } j\leq j
\end{equation}
such that $\calG_j'$ is contained in $\calG_{s(j)}$ and $\calG_{t(j)}$ for each $1\leq j\leq I+J$ and denote the inclusion maps by
\begin{equation}
\label{equ_virt_free_maps2_s_t}
\calG_{s(j)} \overset{\iota_j}{\hookleftarrow} \calG_j' \overset{\kappa_j}{\hookrightarrow} \calG_{t(j)}
\end{equation}
Under these assumptions the data $((\calG_i)_i, (\calG_j')_j, (\iota_j)_j,(\kappa_j)_j)$ constitute a finite graph of groups $\calG_Q$ with vertex groups $(\calG_i)_i$ and edge groups $(\calG_j')_j$ in the sense of \mbox{Bass-Serre} theory. \eqref{equ_virt_free_decomp} says that $\calG$ is isomorphic to the fundamental group $\pi_1(\calG_Q)$ of $\calG_Q$ . For every $\calG$ that arises from finite groups like in \eqref{equ_virt_free_decomp} it is possible to choose alternative embeddings of $\calG_1',\ldots ,\calG_{I+J}'$ into $\calG$ such that there are maps $s,t$ satisfying \eqref{equ_virt_free_maps1_s_t} and \eqref{equ_virt_free_maps2_s_t}, because every finite subgroup of $\calG$ is up to conjugation contained in one of the $\calG_i$ , $0\leq i\leq I$ $-$ more precisely every finite subgroup of an amalgamated free product $\calH *_\calF \calE$ is conjugated to a subgroup of $\calH$ or $\calE$ \footnote{see \cite[§I.4.3, Cor.]{xxtrees}} and every finite subgroup of an HNN extension $\HNN{\calH}{\calF}{\iota,\kappa}$ is conjugated to a subgroup of $\calH$ .\footnote{This is a consequence of \cite[Thm. 4]{xxsubgroupsHNN}.}\par%
Throughout this paper we fix a finitely generated virtually free group $\calG$ , a \mbox{decomposition} \eqref{equ_virt_free_decomp} and maps $s,t$ satisfying \eqref{equ_virt_free_maps1_s_t} and \eqref{equ_virt_free_maps2_s_t}. So the readers may for simplicity use as a definition: A finitely generated group is virtually free if and only if it is of the form \eqref{equ_virt_free_decomp} for finite groups $\calG_i$ , $\calG_j'$ and inclusions \eqref{equ_virt_free_maps2_s_t} among them.\par%
Although the general description \eqref{equ_virt_free_decomp} of our virtually free group $\calG$ might look \mbox{intimidating}, the examples we want to keep in mind are quite down-to-earth. \mbox{Trivially} every finite group is virtually free. Some of the easiest non-trivial examples are the groups $\Gamma_{a,b}:= C_a*C_b$ which by definition are free products of finite cyclic groups $C_a$ and $C_b$ . Prominent examples of this class are the infinite dihedral group $\bbD_\infty =\Gamma_{2,2}$ and $\Gamma_{2,3}$ which is isomorphic to $\mathbf{PSL}_2(\bbZ)$ . We may enlarge this class of examples by picking a common divisor $c$ of $a$ and $b$ as well as embeddings $C_c\hookrightarrow C_a,C_b$ . This gives rise to the virtually free group $C_a *_{C_c} C_b$ . A prominent example here is $C_4*_{C_2} C_6$ which is isomorphic to $\mathbf{SL}_2(\bbZ)$ .\par%
The arithmetic groups $\GL{2}(\bbZ)$ and $\mathbf{PGL}_2(\bbZ)$ are virtually free as well $-$ they arise as $\bbD_4 *_{C_2\times C_2} \bbD_6$ and $\bbD_2 *_{C_2} \bbD_3$ . To define the inclusions of $C_2\times C_2$ and $C_2$ we consider the presentation
\[
\bbD_c = \langle s,t\mid s^2=t^2=1=(st)^c\rangle
\]
of the dihedral group. If $c=2a$ is even, we embed $C_2\times C_2$ into $\bbD_{2a}$ by sending the 2 generators to $s$ and $(st)^a$ . For arbitrary $c$ we embed $C_2$ into $\bbD_c$ by sending the generator to $s$ .\par%
Since every finite index subgroup of a virtually free group is again virtually free, all congruence subgroups of the four above mentioned arithmetic groups are virtually free as well. Another class of examples are of course the free groups: The free group $F_a$ on $a$ generators arises by taking $J=a$ trivial HNN extensions of the trivial group, i.e. in terms of our description \eqref{equ_virt_free_decomp} set $I=0$ and all $\calG_i$ , $\calG_j'$ to be the trivial group.\par%
In \cite[§2]{xxqurves} L. Le Bruyn discusses an analogue of graphs of groups for algebras: For $\calA,\calC$ two $K$-algebras and $\iota,\kappa:\calC\hookrightarrow \calA$ injective $K$-algebra homomorphisms we consider the induced embeddings
\[
\iota',\kappa' :\calC \hookrightarrow \calA *_K K[t,t\inv] \quad, \quad \iota'(f):= \iota(f) \quad ,\quad  \kappa'(f):=t\inv\kappa(f)t
\]
where $*_K$ denotes the coproduct of $K$-algebras. We define the \defit{HNN extension} $\HNN{\calA}{\calC}{\iota,\kappa}$ of $\calA$ by $\calC$ as the coequalizer of
\[
\begin{tikzcd}
\calC \arrow[r,hook,shift left=0.5ex, "{\iota'}"] \arrow[r,hook',shift right=0.5ex,"\kappa'"'] & \calA *_K K[t,t\inv]
\end{tikzcd}
\]
We are mostly interested in HNN extensions of algebras, because they arise as group algebras of HNN extensions of groups: The group algebra functor $K[-]$ is a left adjoint, hence, it preserves colimits. So for every HNN extension of groups we obtain an isomorphism $K[\HNN{\calH}{\calF}{\iota,\kappa}]\cong \HNN{K[\calH]}{K[\calF]}{\iota,\kappa}$ .\footnote{We denote the induced algebra homomorphisms $K[\iota],K[\kappa]: K[\calF]\to K[\calH]$ simply by $\iota$ and $\kappa$ .} Moreover applying the functor $K[-]$ to our decomposition \eqref{equ_virt_free_decomp} we get a $K$-algebra isomorphism between $K[\calG]$ and
\begin{equation}
\label{equ_decomp_groupalg}
\left( \ldots \left(\left( \ldots \left(\left(K[\calG_0] *_{K[\calG_1']} K[\calG_1] \right)*_{K[\calG_2']} K[\calG_2]\right) \ldots \right) \HNN{}{K[\calG_{I+1}']}{\iota_{I+1},\kappa_{I+1}} \right) \ldots \right) \HNN{}{K[\calG_{I+J}']}{\iota_{I+J},\kappa_{I+J}}
\end{equation}
Analogous to \cite[§2]{xxqurves} we say that a $K$-algebra $\calA$ is the \defit{fundamental algebra of a finite graph of finite dimensional semisimple $K$-algebras} if there are $I,J\in\bbN_0$ , maps $s,t:\{1,\ldots ,I+J\}\to \{0,\ldots ,I\}$ fulfilling \eqref{equ_virt_free_maps1_s_t} as well as finite \mbox{dimensional} \mbox{semisimple} \mbox{$K$-algebras} $\calA_0, \ldots ,\calA_I$ and $\calA_1',\ldots ,\calA_{I+J}'$ and $K$-algebra \mbox{embeddings} \mbox{$\iota_j: \calA_j' \hookrightarrow \calA_{s(j)}$ ,} $\kappa_j: \calA_j' {\hookrightarrow} \calA_{t(j)}$ such that $\calA$ is isomorphic to
\begin{equation}
\label{equ_fundamental_algebra_of_semisimples}
\left( \ldots \left(\left( \ldots \left(\left(\calA_0 *_{\calA_1'} \calA_1 \right)*_{\calA_2'} \calA_2\right) \ldots \right) \HNN{}{\calA_{I+1}'}{\iota_{I+1},\kappa_{I+1}} \right) \ldots \right) \HNN{}{\calA_{I+J}'}{\iota_{I+J},\kappa_{I+J}}
\end{equation}
The group algebra $K[\calG]$ of the finitely generated virtually free group $\calG$ is the fundamental algebra of a finite graph of finite \mbox{dimensional} semisimple $K$-algebras whenever $\chara{K}$ is not a prime number dividing the order of one of the finite groups $\calG_i$ , $0\leq i\leq I$ , since $K[\calG]$ is isomorphic to \eqref{equ_decomp_groupalg}.\par%
The main reason we are studying virtually free groups in this paper is the following fact:\footnote{see \cite[Thm. 1]{xxhereditarydicks}} If $\calH$ is a finitely generated group and $K$ a field, then the group algebra $K[\calH]$ is (left) hereditary\footnote{Since every group is isomorphic to its opposite, left and right hereditaryness are equivalent.} if and only if $\calH$ is virtually free and contains no elements of order $\chara{K}$ . Using the decomposition \eqref{equ_virt_free_decomp} one can show that $K[\calG]$ is hereditary if and only if $\chara{K}$ does not divide the orders of the groups $\calG_0,\ldots ,\calG_I$ , because every finite subgroup $\calF\subseteq \calG$ (i.e. in particular for $\calF$ cyclic of prime order) is conjugated to a subgroup of one of the groups $(\calG_i)_i$ .\par%
Recall that a $K$-algebra $\calA$ is called \defit{formally smooth} if its Hochschild \mbox{cohomology} $H\!H^a(\calA,-)$ vanishes in degree $a\geq 2$ . This is equivalent to $\calA$ satisfying a \mbox{lifting} property along square-zero extensions of $K$-algebras.\footnote{see \cite[Prop. 9.3.3]{xxweibel}} Every formally smooth \mbox{$K$-algebra} is left and right hereditary\footnote{use e.g. \cite[Lemma 9.1.9]{xxweibel}} and the fundamental algebra of a finite graph of finite dimensional semisimple $K$-algebras is formally smooth.\footnote{see \cite[Thm. 1]{xxqurves}} So a group algebra $K[\calH]$ of a finitely generated group $\calH$ is formally smooth if and only if it is hereditary, i.e. if and only if $\calH$ is virtually free and contains no elements of order $\chara{K}$ .\par%
For (parts of) the machinery of this paper to work it is crucial that $K[\calG]$ is formally smooth, i.e. $\chara{K}$ has to be zero or a suitable prime. To make things more convenient we will moreover assume that $K$ is large enough which brings us to the notion of suitable fields:\par%
Let $\calC$ be a finite dimensional semisimple $K$-algebra.\footnote{We are mostly interested in the case $\calC = K[\calF]$ for $\calF$ a finite group of order coprime to \mbox{$\chara{K}$ .}} By Artin-Wedderburn theory we know that $\calC$ is (isomorphic to) a product of matrix algebras
\begin{equation*}
\bfM_{\delta_1}(\calD_{1}) \times \ldots \times \bfM_{\delta_c}(\calD_{c})
\end{equation*}
with $\calD_1,\ldots ,\calD_c$ finite dimensional $K$-division algebras. We say that $\calC$ is \defit{completely split} if all simple left $\calC$-modules are absolutely simple or {equivalently} if $\calD_\gamma\cong K$ for all $1\leq \gamma\leq c$ , i.e. $\calC$ is completely split if and only if it is of the form
\begin{equation}
\label{equ_normalform_semisimple}
K^{c_1} \times \bfM_2(K)^{c_2} \times \ldots \times \bfM_e(K)^{c_e}
\end{equation}
for non-negative integers $e,c_\epsilon\in \bbN_0$ . Note that all left $\calC \otimes_K F$-modules for every field extension $F\supseteq K$ are defined over $\calC$ . We say that a field $K$ is of \defit{\mbox{suitable} \mbox{characteristic}} for the virtually free group $\calG$ if $\calG$ contains no \mbox{elements} of order $\chara{K}$ . We call a field $K$ \defit{suitable} for $\calG$ if it is perfect, of suitable \mbox{characteristic} and $K[\calF]$ is completely split for every finite subgroup $\calF\subseteq \calG$ .\par%
Note that being suitable is a relative notion $-$ it depends on which virtually free group it refers to. The readers may convince themselves that every algebraic field extension of a suitable field is again suitable and (using that every finite subgroup of $\calG$ is contained in a $\calG_i$ up to conjugation) that every perfect field of suitable characteristic admits a finite extension that is suitable.\par%
Within this paper we will study representations of finitely generated virtually free groups over suitable finite fields. However, all of our methods apply to the more general case of representations of the fundamental algebra \eqref{equ_fundamental_algebra_of_semisimples} of a finite graph of finite dimensional semisimple $\bbF_q$-algebras under the assumption that each of the semisimple algebras $\calA_i$ and $\calA_j'$ are {completely split}.
\subsection{Geometric methods}
\label{subsec_geometric_methods}
The algebro-geometric methods in this paper are \mbox{written} in the language of schemes. Since we will work almost entirely with (affine) schemes of finite type over a perfect field $K$ , those readers who are less \mbox{comfortable} with schemes may instead think of the associated $\overline{K}$-varieties and see $K$-valued points as fixed points of the natural Galois action of $\Aut[K]{\overline{K}}$ , connected/irreducible \mbox{components} as orbits of the natural Galois action on the connected/irreducible \mbox{components,} etc. For this whole subsection fix a field $K$ .
\fakesubsubsection{Counting rational points}
Let $C$ be a commutative ring and $X$ be a $C$-scheme. For each commutative \mbox{$C$-algebra} $B$ , we will denote by $X(B)$ the set of \defit{$B$-valued points} of $X$ , i.e. the set of $C$-scheme morphisms $\spec{B}\to X$ . If $C=K$ is a field, we also use the term rational points for the \mbox{$K$-valued} points $X(K)$ .\par%
Now assume $C$ is of finite type over $\bbZ$ and $X$ is separated and of finite type over $C$ . A polynomial $P\in\bbZ[s]$ is called \defit{counting polynomial} of $X$ if for every ring homomorphism $C\to \bbF_q$ to a finite field $\#X(\bbF_q)=P(q)$ . We say that $X$ is \defit{polynomial count} if $X$ admits a counting polynomial.
\begin{bsp}\thmabsatz
The general linear group (scheme) $\GL{d}$ is polynomial count, its counting polynomial is given by $\countgl[d] := \prod_{\delta=0}^{d-1} (s^d -s^\delta)$ .
\end{bsp}
Note that counting polynomials are unique and that the reduction $X_{\text{red}}$ of a \mbox{polynomial} count scheme $X$ is again polynomial count with the same counting \mbox{polynomial.} We will need the following two facts on polynomial count schemes:
\begin{compactenum}[(i)]
\item If $P\in \bbQ(s)$ is a rational function and $\#X(\bbF_q)=P(q)$ for each homomorphism $C\to \bbF_q$ , then $P$ lies in the subring $\bbZ[s]$ (and is a counting polynomial of $X$ ).
\item If $C$ is a subring of $\bbC$ and $X$ is a polynomial count $C$-scheme with counting polynomial $P$ , then $P(xy)\in \bbZ[x,y]$ is the E-polynomial\footnote{For the definition of E-polynomials see e.g. \cite[Appendix]{xxhauselRodriguez}. {The notion of E-polynomials only occurs as an application/motivation, readers only interested in the counting of representations over finite fields may ignore it.}} of the analytification $X(\bbC)= (X\times_C \spec{\bbC})^{\text{an}}$ and $P(1)$ is the Euler characteristic of $X(\bbC)$ .
\end{compactenum}
For proofs of the above facts see e.g. \cite[Appendix]{xxhauselRodriguez} and \cite[§6]{xxreinekecounting}.
\fakesubsubsection{Geometry of representation spaces}
The most important schemes we discuss in this paper arise from the \mbox{representation} spaces of algebras, whose construction we will now recall. Recall that the \mbox{functor} $\bfM_d: \katcalg{K} \to \katalg{K}$ which sends $C$ to the matrix algebra $\bfM_d(C)$ is a right \mbox{adjoint.} We denote its left \mbox{adjoint} by $\calR_d$ . For $\calA$ a finite type $K$-algebra we call \mbox{$\operatorname{Rep}_d(\calA):=\spec{\calR_d(\calA)}$} the \defit{$d$-th representation space} of $\calA$ . $\operatorname{Rep}_d(\calA)$ is an affine finite type $K$-scheme \mbox{admitting} a natural bijection
\begin{equation}
\label{equ_functor_of_pts_Repspace}
\operatorname{Rep}_d(\calA)(C) \cong \moralg{K}{\calA,\bfM_d(C)}
\end{equation}
for every commutative $K$-algebra $C$ . Denote the image of $x\in\operatorname{Rep}_d(\calA)(C)$ under \eqref{equ_functor_of_pts_Repspace} by $\rho_x$ . The right hand side of \eqref{equ_functor_of_pts_Repspace} admits a natural $\GL{d}(C)$-action via conjugation for each $C$ , hence, the general linear group (scheme) $\GL{d,K}$ acts on $\operatorname{Rep}_d(\calA)$ in terms of a $K$-scheme morphism $\sigma: \GL{d,K} \times_K \operatorname{Rep}_d(\calA) \to \operatorname{Rep}_d(\calA)$ .\par%
For a group scheme action we have two notions of the orbit of a point: If $x$ is a $C$-valued point of $\operatorname{Rep}_d(\calA)$ , then we have the orbit $\GL{d}(C).x\subseteq \operatorname{Rep}_d(\calA)(C)$ $-$ which we call the \defit{set-theoretic orbit} of $x$ $-$ as well as the \defit{algebro-geometric orbit} \mbox{$\orbit{x}$ .} The latter is defined as the image of the orbit map
\[
\orbitmap{x} := (\sigma\circ(\id[{\GL{d,K}}]\times x),\pr_2): \GL{d,K} \times_K \spec{C} \to X\times_K\spec{C}
\]
{If $x\in\operatorname{Rep}_d(\calA)(F)$ is an $F$-valued point for $F\supseteq K$ a finite field extension, then $\orbit{x}\subseteq\operatorname{Rep}_d(\calA)$ is locally closed and we may consider it as a reduced \mbox{locally} closed \mbox{subscheme.}} The two notions of orbits are closely related to each other: If $x\in \operatorname{Rep}_d(\calA)(K)$ is a \mbox{$K$-valued} point, $F\supseteq K$ a field extension and $x'$ the \mbox{$F$-valued} point \mbox{associated} to $x$ via pulling it back along $\spec{F}\to \spec{K}$ , then \mbox{$\orbit{x}(F)=\GL{d}(F).x'$ .}\footnote{In general the inclusion $\subseteq$ is wrong for group scheme actions. For representation spaces it holds, because representations of an algebra have no \textit{twisted forms}, i.e. if $\calM,\calN$ are left $\calA$-modules and 
$\calM\otimes_K F\cong \calN\otimes_K F$ for some field extension $F\supseteq K$ , then $\calM$ and $\calN$ are already isomorphic.}\par%
The purpose of the action on $\operatorname{Rep}_d(\calA)$ is that $K$-valued points $x,y\in \operatorname{Rep}_d(\calA)(K)$ are in the same (set-theoretic) orbit if and only if the representations $\rho_x$ and $\rho_y$ are isomorphic, i.e. there is a natural bijection $\iso[d](\calA) \cong \nicefrac{\operatorname{Rep}_d(\calA)(K)}{\GL{d}(K)}$ .\par%
Within this paper we will usually not distinguish strictly between the set-theoretic orbit $\GL{d}(K).x$ of a $K$-valued point $x$ , the isomorphism class of the \mbox{corresponding} representation $\rho_x:\calA \to \bfM_d(K)$ and the isomorphism class of its associated left module, which we denote by $\calM_x$ . Given a left $\calA$-module $\calM$ we denote the \mbox{$K$-valued} point corresponding to $\calM$ under a given choice of basis by $x_\calM$ . If $\calM$ is the left \mbox{module} associated to the point $x$ , we will also sometimes denote the \mbox{algebro-geometric} orbit $\orbit{x}$ by $\orbit{\calM}$ .\par%
Similar to the algebro-geometric orbits one defines the \defit{stabilizer} $\stab{x}$ of a $\calC$-valued point $x\in \operatorname{Rep}_d(\calA)(\calC)$ geometrically as the fibre product defined by the pullback square
\[
\begin{xy}\xymatrix{
\stab{x} \ar[r]\ar[d] & \spec{C}\ar[d]^{(x,\id)}
\\
\GL{d,K} \times_K \spec{C} \ar[r]_{\orbitmap{x}} & \operatorname{Rep}_d(\calA) \times_K \spec{C}
}\end{xy}
\]
$\stab{x}$ is a closed $C$-subgroup scheme of $\GL{d,K} \times_K \spec{C}$ . Its $B$-valued points for any commutative $C$-algebra $B$ are given by
\begin{equation}
\label{equ_stabilizer_Aut_fctor_of_pts}
\stab{x}(B) \cong \{g\in \GL{d}(B) \mid g.x' = x'\} = \Aut[\calA\otimes_K B]{\calM_x\otimes_C B}
\end{equation}
where $x'\in\operatorname{Rep}_d(\calA)(B)$ is the $B$-valued point associated to $x$ . So $\stab{x}(B)$ is nothing but the (set-theoretic) stabilizer subgroup in $\GL{d}(B)$ of the point $x'$ .\par%
%
%
%
%
If $\varphi:\calA \to \calB$ is a homomorphism of finite type $K$-algebras, then functoriality gives us an induced $K$-scheme morphism $\operatorname{Rep}_d(\calB)\to\operatorname{Rep}_d(\calA)$ for each $d$ which realizes the restriction of scalars functor geometrically. We denote this morphism by $\varphi^*$ . In fact \ifthenelse{1=1}{$\operatorname{Rep}_d(-)$}{$\operatorname{Rep}_d(-): \katalg{K} \to \left(\katrelaffsch{K}\right)\op $} is a left adjoint functor from (finite type) $K$-algebras to the opposite category of affine (finite type) $K$-schemes, hence, it maps colimits of (finite type) $K$-algebras to limits of affine (finite type) $K$-schemes. For example $\operatorname{Rep}_d(\calA *_K\calB)\cong \operatorname{Rep}_d(\calA) \times_K \operatorname{Rep}_d(\calB)$ for $\calA,\calB$ two finite type $K$-algebras. Moreover \eqref{equ_functor_of_pts_Repspace} shows that
\begin{equation*}
\operatorname{Rep}_d(\calA\otimes_K F) \cong \operatorname{Rep}_d(\calA)\times_K \spec{F}
\end{equation*}
are naturally isomorphic $F$-schemes for all field extensions $F\supseteq K$ .\par%
The geometry of representation spaces and their orbits play a substantial role within this paper. We will frequently make use of the following fundamental facts:\footnote{See \cite[§2.3]{xxkirillov} for {(v) and (vi)} in the case $\calA=\bbC\vec{Q}$ the path algebra of a quiver. The proofs hold in our setting without significant change.}
\begin{compactenum}[(iii)]
\item If $\calA$ is formally smooth, then $\operatorname{Rep}_d(\calA)$ is {a regular scheme}.\footnote{$\operatorname{Rep}_d(\calA)$ is formally smooth over $K$ in the sense of \cite[Tag 02H0]{xxstacks} because of the lifting property of $\calA$ . So by \cite[Tags 02H6]{xxstacks} it is smooth over $K$ , hence, regular by \mbox{\cite[Tag 056S]{xxstacks}.}}
\item[(iv)] If $x\in\operatorname{Rep}_d(\calA)(F)$ is an $F$-valued point for a field extension $F\supseteq K$ , then $\orbit{x}$ is irreducible and in particular connected.\footnote{It is the continuous image of the general linear group $\GL{d,F}$ .}
\item[(v)] If $K$ is perfect and $0\to \calN\to \calW\to \calM\to 0$ a short exact sequence of finite dimensional left $\calA$-modules, then $\orbit{\calN\oplus\calM}\subseteq \overline{\orbit{\calW}}$ .
\item[(vi)] If $K$ is perfect, then an orbit $\orbit{\calM}\subseteq \operatorname{Rep}_d(\calA)$ is closed if and only if the corresponding left $\calA$-module $\calM$ is semisimple.
\end{compactenum}
\fakesubsubsection{E-polynomials of moduli spaces of representations}
\label{fakesubsub_E-polynomials}
Since the isomorphism classes of representations of $\calA$ are parametrized by orbits of representation spaces, it is natural to define moduli spaces of representations of $\calA$ in terms of quotients of representation space. We denote the \defit{GIT quotient}\footnote{see e.g. \cite[§1.2, Thm. 1.1]{xxGIT}} of $\operatorname{Rep}_d(\calA)$ by
\begin{equation*}
\modulirep[]{\calA}{d}:= \operatorname{Rep}_d(\calA)/\!\!/\GL{d,K} = \spec{\calR_d(\calA)^{\GL{d,K}}}
\end{equation*}
If the field $K$ is finite or algebraically closed,\footnote{More generally it would be sufficient to require that $K$ is perfect and has trivial Brauer group.} then there is a natural bijection \mbox{$\modulirep{\calA}{d}(F)\cong \ssimp[d](\calA\otimes_K F)$} for every algebraic field extension $F\supseteq K$ . So for such $K$ we call $\modulirep{\calA}{d}$ the \defit{(GIT) moduli space of $d$-dimensional semisimple \mbox{representations}} of $\calA$ . It contains a (possibly empty) open subscheme $\modulirep[absim]{\calA}{d}$ which for $K$ as above admits a natural bijection $\modulirep[absim]{\calA}{d}(F) \cong \absimp[d](\calA \otimes_K F)$ for every algebraic field extension $F\supseteq K$ . Accordingly we call $\modulirep[absim]{\calA}{d}$ the \defit{(GIT) moduli space of $d$-dimensional absolutely simple representations} of $\calA$ .\par%
Note that for $K=\bbF_q$ a finite field the counting functions $\countfct[ss]{,\calA}{d}$ and $\countfct[absim]{,\calA}{d}$ count the rational points of these GIT moduli spaces. So whenever counting polynomials $\countpol[ss]{}{d}$ and $\countpol[absim]{}{d}$ as in Theorem \ref{thm_state_of_the_art} exist, they are in fact counting polynomials of these GIT moduli spaces.\par%
Now assume that $\calA = \calA'\otimes \bbF_q := \calA'\otimes_\bbZ \bbF_q$ is defined over $\bbZ$ by a finite type $\bbZ$-algebra \mbox{$\calA'$ .}\footnote{For example $\calA'=\bbZ[\calG]$ for $\calG$ a finitely generated group or $\calA'=\bbZ\vec{Q}$ for $\vec{Q}$ a finite quiver.} One can define representation spaces and GIT moduli spaces of $\calA'$ as \mbox{$\bbZ$-schemes} using Seshadri's generalization of geometric invariant theory.\footnote{see \cite{xxseshadriGIT}} In this way we obtain $\bbZ$-schemes $\modulirep{\calA'}{d}$ and $\modulirep[absim]{\calA'}{d}$ such that for $F=\bbC$ and $F=\overline{\bbF_p}$ for all primes $p$ in an open subset of $\spec{\bbZ}$ we have\footnote{see \cite[Appendix B, Thm. B.3]{xxcrawleyVan}}
\begin{equation*}
\resizebox{15cm}{!}{$
\modulirep{\calA'}{d}\times \spec{F}\cong \modulirep{\calA'\otimes F}{d}
\; , \;
\modulirep[absim]{\calA'}{d}\times \spec{F}\cong \modulirep[absim]{\calA'\otimes F}{d}
$}
\end{equation*}
Hence, using {(i) from Subsection \ref{subsec_geometric_methods}} we see that if there are rational functions $\countpol[ss]{}{d}, \countpol[absim]{}{d}\in\bbQ(s)$ which satisfy \eqref{equ_state_of_the_art}, they must already belong to $\bbZ[s]$ . Furthermore using {(ii) from Subsection \ref{subsec_geometric_methods}} we see that whenever the counting polynomials exist, the E-polynomials of $\modulirep{\calA' \otimes \bbC}{d}^{\text{an}}$ and $\modulirep[absim]{\calA' \otimes \bbC}{d}^{\text{an}}$ are given by $\countpol[ss]{}{d}(xy)$ and $\countpol[absim]{}{d}(xy)$ .\par%
When $\calA'=\bbZ[\calG]$ is the group algebra of a finitely generated group $\calG$ , (the analytification of) the moduli space $\modulirep{\bbC[\calG]}{d}$ is also called the \defit{$\GL{d}(\bbC)$-character variety} of $\calG$ and denoted by $X_\calG(\GL{d}(\bbC))$ . Since our methods enable us to compute the counting polynomials explicitly (e.g. using the accompanying SageMath code), we obtain a new approach to determine the E-polynomials of $\GL{d}(\bbC)$-character varieties\footnote{In fact we can more generally compute the E-polynomials of the connected components of the character varieties individually.} of virtually free groups.\par%
\fakesubsubsection{Associated fibre spaces and special groups}
We now want to recall the construction of associated fibre spaces.\footnote{Most of the facts we are collecting here can e.g. be found in \cite{xxserreFibreBundles}.} Let $G$ be a linear algebraic group over $K$ , $H\subseteq G$ a closed subgroup and $X$ an affine $K$-scheme endowed with an $H$-action. We define an induced $H$-action on $G\times_K X$ via the natural transformation
\[
H(C)\times G(C)\times X(C)\to G(C)\times X(C) \quad,\quad h.(g,x):=(gh\inv, h.x)
\]
for any commutative $K$-algebra $C$ . This is a free action and its respective quotient\footnote{In the case of free actions we will denote the GIT quotient with a single $/$ instead of $/\!\!/$ .} \mbox{$G\times^H X:=\nicefrac{G\times_K X}{H}$} is called the \defit{associated $G$-fibre space.}\par%
If $(g,x)\in G(C)\times X(C)$ is a $C$-valued point, we denote its image in $(G\times^H X)(C)$ by $g*x$ . We have a natural morphism $X\to G\times^H X$ given by $x\mapsto 1*x$ . If $Y$ is a $K$-scheme with $G$-action and $\varphi:X\to Y$ is an $H$-equivariant morphism, then we obtain a unique $G$-equivariant morphism $\varphi'$ such that $\varphi$ factorizes as
\begin{equation}
\label{equ_extension_to_assoc_fibre}
X \to G\times^H X \overset{\varphi'}{\to} Y
\end{equation}
For $G$ and $H$ (geometrically) irreducible we have that $G\times^H X$ is irreducible/connected if and only if $X$ is irreducible/connected. Moreover we have the following useful lemma.
\begin{lem}\footnote{see \cite[§II.3.7, Lemma 4]{xxslodowy80}}\thmabsatz
\label{lem_slodowy_trick}
Let $\varphi: Y\to \nicefrac{G}{H}$ be a $G$-equivariant morphism. If $e\in \nicefrac{G}{H}(K)$ is the $K$-point lying under the unit of $G$ and $X:=\varphi\inv (e)$ with inclusion map $\iota: X\to Y$ , then the induced map $\iota': G\times^H X \to Y$ from \eqref{equ_extension_to_assoc_fibre} is a $G$-equivariant isomorphism.
\end{lem}
We also want to recall the notion of special algebraic groups. A linear algebraic group $G$ over $K$ is \defit{special} if the quotient map $\pi: X\to \nicefrac{X}{G}$ for any affine finite type $K$-scheme $X$ with a free $G$-action is Zariski-locally a trivial bundle, i.e. there is an open covering $\nicefrac{X}{G} = \bigcup_\alpha U_\alpha$ such that for each $\alpha$ there is a $G$-equivariant isomorphism $\Phi_\alpha$ admitting a commutative diagram\footnote{This is equivalent to the ordinary definition by \cite[§4.3, Théorème 2]{xxserreFibreBundles} and by the fact that all such quotient maps are principal fibre bundles (see \cite[§0.4, Prop. 0.9]{xxGIT}).}
\[
\begin{xy}\xymatrix{
\pi\inv(U_\alpha) \ar[rr]^{\Phi_\alpha} \ar[rd]_\pi && G\times_K U_\alpha\ar[ld]^{\pr_2}
\\
& U_\alpha
}\end{xy}
\]
\begin{bsp}\thmabsatz
\label{bsp_special_stabilizer}
Let $\calC$ be a finite dimensional semisimple $K$-algebra which is \mbox{completely} split, i.e. $\calC$ is of the form \eqref{equ_normalform_semisimple}. Denote its (absolutely) simple left modules by $\calL_0,\ldots ,\calL_{c-1}$ . Each finite dimensional left $\calC$-module $\calM$ is semisimple, i.e.
\[
\calM \cong \calL_0^{\oplus m(0)} \oplus \ldots \oplus \calL_{c-1}^{\oplus m(c-1)}
\]
for some $m\in{\bbN_0}^c$ . Using the bijection \eqref{equ_stabilizer_Aut_fctor_of_pts} and $\End[\calA]{\calL_\gamma}=K$ we obtain an isomorphism $\stab{x_\calM}(B)\cong \Aut[\calC\otimes_K B]{\calM\otimes_K B}\cong \GL{m(0)}(B) \times \ldots \times \GL{m(c_1)}(B)$ for every \mbox{commutative} $K$-algebra $B$ . Hence, $\stab{x_\calM} \cong \GL{m(0),K} \times \ldots \times \GL{m(c-1),K} $ which is a special linear algebraic group.\footnote{This can be seen e.g. from the classification of special groups in \cite{xxgrothendieckSpecial}.}
\end{bsp}
If $H$ is a special linear algebraic group acting freely on an affine finite type \mbox{$K$-scheme} $X$ with quotient $\nicefrac{X}{H}$ , then the canonical injection
\begin{equation}
\label{equ_special_bijection_rational_points}
\nicefrac{X(F)}{H(F)} \to \left(\nicefrac{X}{H}\right)(F)
\end{equation}
is bijective for every field extension $F\supseteq K$ . This in particular applies to the case where $\nicefrac{X}{H}$ is an associated fibre space $G\times^H X$ .
%
%
%
\section{Some invariants of virtually free groups}
\label{section_3_some_invariants_of_virtually_free_groups}
\noindent For the whole section fix a perfect field $K$ .
\subsection{Dimension vectors}
\label{subsec_dimension_vectors}
We will now associate to every finite type $K$-algebra $\calA$ a commutative monoid\footnote{i.e. a set $M$ with a binary operation $+:M\times M\to M$ which is associative, commutative and admits a neutral element $0$} $\calT(\calA)$ and a monoid homomorphism $\mongrad: \calT(\calA)\to \bbN_0$ which generalizes the dimension vector monoid from quiver representation theory.\footnote{see Example \ref{bsp_dimvec}(b) below} For $d\in\bbN_0$ we denote by $\calT_d(\calA)$ the set of connected components $Z\subseteq \operatorname{Rep}_d(\calA)$ containing a rational point, i.e. $Z(K)\neq \emptyset$ . As a set we define $\calT(\calA)$ as the disjoint union
\[
\calT(\calA):= \bigsqcup_{d\geq 0} \calT_d(\calA)
\]
and we define the map $\mongrad$ via $\vert \calT_d(\calA)\vert=d$ . To define the monoid structure on $\calT(\calA)$ we consider the direct sum map
\[
\oplus_{c,d} : \operatorname{Rep}_c(\calA)\times_K \operatorname{Rep}_d(\calA)\to \operatorname{Rep}_{c+d}(\calA)
\]
For $(Z,Z')\in \calT_c(\calA)\times \calT_d(\calA)$ the product $Z\times_K Z'$ is connected,\footnote{see \cite[Tags 0385 \& 04KV]{xxstacks}} hence, there is a unique connected component $Z+Z'\in\calT_{c+d}(\calA)$ containing $\oplus_{c,d}(Z\times_K Z')$ .
\par%
The monoid $\calT(\calA)$ has been studied in the past under other names like {component semigroup}.\footnote{see e.g. \cite[§4]{xxqurves}} Since we want to emphasize the analogy to dimension vectors, we will refer to it as the \defit{dimension vector monoid} of $\calA$ and call the elements of $\calT_d(\calA)$ \defit{dimension vectors of total dimension $d$ .} Usually we will think of dimension vectors as abstract monoid elements. Whenever we want to refer to the connected component (associated to) $m$ as a geometric object, we will denote it by $\operatorname{Rep}_m(\calA)$ .\par%
Since the orbit $\orbit{\calM}$ associated to a left $\calA$-module $\calM$ is connected, it belongs to a unique connected component. Denote the corresponding dimension vector by $\dimvec(\calM)$ .\par%
The dimension vector monoid $\calT(\calA)$ is contravariant functorial in $\calA$ : If \mbox{$\varphi:\calA\to \calB$} is a $K$-algebra homomorphism and $m\in\calT(\calB)$ a dimension vector, then denote by $\calT(\varphi)(m)$ the dimension vector associated to the connected component which contains the image of $\operatorname{Rep}_m(\calB)$ under $\varphi^*$ . This defines a monoid homomorphism \mbox{$\calT(\varphi):\calT(\calB)\to \calT(\calA)$ .}\par%
The homomorphism $\calT(\varphi)$ induces maps $\calT_d(\varphi):\calT_d(\calB)\to \calT_d(\calA)$ for all $d\in\bbN_0$ , because restriction of scalars preserves the vector space dimension of modules. Since every bijective monoid homomorphism is an isomorphism, $\calT(\varphi)$ is an isomorphism if and only if the map $\calT_d(\varphi)$ is bijective for every $d$ .
\begin{bsp}\mbo
\label{bsp_dimvec}
\startabc
\item Let $\calC$ be a finite dimensional semisimple $K$-algebra. We assume that $\calC$ is \mbox{completely} split, i.e. of the form \eqref{equ_normalform_semisimple}. All left $\calC\otimes_K F$-modules for every field extension $F\supseteq K$ are defined over $\calC$ . Therefore the (finitely many) orbits of the $K$-valued points cover $\operatorname{Rep}_d(\calC)$ and all of them are connected and closed by {(iv) and (vi) in Subsection \ref{subsec_geometric_methods}}. We deduce that the orbits of the $K$-valued points are nothing but the connected components and that the map \mbox{$\iso[d](\calC)\to \calT_d(\calC)$ ,} $[\calM]\mapsto \dimvec(\calM)$ is bijective for every $d\in\bbN_0$ . Since there is precisely one simple left $\calC$-module for every matrix algebra factor in \eqref{equ_normalform_semisimple}, we have established a monoid isomorphism
\begin{equation*}
\calT(\calC) \cong {\bbN_0}^{c_1 +\ldots + c_e}
\end{equation*}
In fact, $\calT(\calC)$ is nothing but the submonoid of the Grothendieck group $G_0(\calC)$ generated by the equivalence classes of the (absolutely) simples in this situation. If $m_\gamma\in {\bbN_0}^{c_1 +\ldots + c_e}$ is the $\gamma$-th standard basis vector for $0\leq \gamma < c_1 +\ldots + c_e$ , then $\vert m_\gamma\vert = c_\epsilon$ for the unique $1\leq \epsilon\leq e$ with \mbox{$c_1 + \ldots + c_{\epsilon-1} < \gamma \leq c_1 + \ldots + c_{\epsilon}$ .}
\item Let $K\vec{Q}$ be the path algebra of a finite quiver $\vec{Q}$ with vertex set {$\mathsf{v} \big(\vec{Q}\big)$} . Applying (a) to the subalgebra $\calC$ spanned by the paths of length zero, one can use Lemma \ref{lem_slodowy_trick} to establish a monoid isomorphism
\[
\calT \big(K\vec{Q}\big) \cong \calT(\calC)\cong {\bbN_0}^{\knotensub[\vec{Q}]}
\]
\stopabc
\end{bsp}
The following proposition gives a complete description of the dimension vector monoid of the group algebra $K[\calG]$ of a finitely generated virtually free group $\calG$ over a suitable field $K$ .
\begin{sa}\thmabsatz
\label{sa_dimvec_virt_free}
Let $\calA$ be a finite type $K$-algebra, $\calB$ and $\calC$ completely split finite dimensional semisimple $K$-algebras and $\varphi_1: \calC\to \calB$ , $\varphi_2,\varphi_3:\calC\to \calA$ $K$-algebra homomorphisms.
\startabc
\item Consider the $K$-algebra pushout $\calA *_\calC \calB$ given by $\varphi_1,\varphi_2$ . The commutative square
\begin{equation}
\label{equ_dimvec_pullback_square}
\begin{xy}
\xymatrix{
\calT(\calA*_\calC \calB) \ar[r] \ar[d] & \calT(\calB) \ar[d]^{\calT(\varphi_1)}
\\
\calT(\calA) \ar[r]_{\calT(\varphi_2)} & \calT(\calC)
}
\end{xy}
\end{equation}
is a {pullback square of commutative monoids}.
\item If $\iota_\calA:\calA \to \calA *_K K[t,t\inv]$ is the canonical $K$-algebra embedding {and\\ {$\Phi: \calA *_K K[t,t\inv] \to \calA *_K K[t,t\inv]$} the $K$-algebra automorphism \mbox{$\Phi(f):=t\inv f t$ ,}} then $\calT(\iota_\calA)$ is an isomorphism and $\calT(\Phi)=\id$ .
\item Consider the HNN extension $\HNN{\calA}{\calC}{\varphi_2,\varphi_3}$ given by $\varphi_2,\varphi_3$ . The diagram
\begin{equation}
\label{equ_dimvec_equalizer}
\calT(\HNN{\calA}{\calC}{\varphi_2,\varphi_3})\to \calT(\calA) \doppelpfeil{\calT(\varphi_2)}{\calT(\varphi_3)} \calT(\calC)
\end{equation}
is an equalizer diagram of commutative monoids.
\item If $\calG$ is the finitely generated virtually free group given by \eqref{equ_virt_free_decomp} and $K$ is a suitable field for $\calG$ , then $\calT(K[\calG])$ is given by
\begin{equation}  
\label{equ_dimvec_total}
\left\{
(m_i)_i \in \prod_{i = 0}^I \calT(K[\calG_i]) \mid \fa 1\leq j\leq I+J: \calT(\iota_{j})(m_{s(j)}) = \calT(\kappa_{j})(m_{t(j)}) 
\right\}
\end{equation}
\stopabc
\end{sa}
\begin{proof}
About a): \eqref{equ_dimvec_pullback_square} induces a homomorphism $\theta: \calT(\calA*_\calC \calB) \to \calT(\calA)\times_{\calT(\calC)} \calT(\calB)$ with
\[
\calT(\calA)\times_{\calT(\calC)} \calT(\calB) = \{(m,n)\in \calT(\calA)\times \calT(\calB) \mid \calT(\varphi_2)(m)=\calT(\varphi_1)(n)\}
\]
$\theta$ is an isomorphism if and only if its restriction \mbox{$\theta_d: \calT_d(\calA*_\calC \calB) \to \calT_d(\calA)\times_{\calT_d(\calC)} \calT_d(\calB)$} is bijective for every $d\in\bbN_0$ . Denote the natural homomorphisms $\calA\to \calA*_\calC \calB$ and $\calB \to \calA*_\calC \calB$ by $\iota_\calA$ and $\iota_\calB$ and the connected components of $\operatorname{Rep}_d(\calA)$ and $\operatorname{Rep}_d(\calB)$ by $X_0,\ldots , X_a$ and $Y_0,\ldots Y_b$ respectively.\par%
Since the contravariant functor $\operatorname{Rep}_d(-)$ maps colimits to limits, we have a natural isomorphism $\operatorname{Rep}_d(\calA *_\calC \calB)\cong \operatorname{Rep}_d(\calA)\times_{\operatorname{Rep}_d(\calC)} \operatorname{Rep}_d(\calB)$ . This yields a decomposition
\[
\operatorname{Rep}_d(\calA *_\calC \calB) = \bigsqcup_{\substack{0\leq \alpha\leq a,\\ 0\leq \beta\leq b}} X_\alpha \times_{\operatorname{Rep}_d(\calC )} Y_\beta
\]
into open and closed subsets. Hence, each connected component $\operatorname{Rep}_v(\calA*_\calC \calB)$ associated to a dimension vector $v\in \calT(\calA *_\calC \calB)$ lies in a unique $X_\alpha \times_{\operatorname{Rep}_d(\calC )} Y_\beta$ and one checks that the homomorphism $\theta$ is given by $\theta_d(v)=(m_\alpha,n_\beta)$ where $m_\alpha\in\calT(\calA)$ and $n_\beta\in\calT(\calB)$ are the dimension vectors (associated to) $X_\alpha$ and $Y_\beta$ . We claim that $\operatorname{Rep}_{m_\alpha}(\calA) \times_{\operatorname{Rep}_d(\calC )} \operatorname{Rep}_{n_\beta}(\calB) = X_\alpha \times_{\operatorname{Rep}_d(\calC )} Y_\beta$ is connected and contains a rational point for each $(m_\alpha,n_\beta)\in \calT(\calA)\times_{\calT(\calC)} \calT(\calB)$ which proves that $\theta_d$ is bijective.\par%
Denote by $u:=\calT(\varphi_2)(m_\alpha)=\calT(\varphi_1)(n_\beta)\in \calT(\calC)$ the dimension vector lying below $(m_\alpha,n_\beta)$ . By construction both $\operatorname{Rep}_{m_\alpha}(\calA)$ and $\operatorname{Rep}_{n_\beta}(\calB)$ map into the connected component $Z:=\operatorname{Rep}_{u}(\calC)$ and we obtain an isomorphism $X_\alpha \times_{\operatorname{Rep}_d(\calC )} Y_\beta \cong X_\alpha \times_Z Y_\beta$ . Since $m_\alpha$ and $n_\beta$ are dimension vectors, there are $K$-valued points $x\in\operatorname{Rep}_{m_\alpha}(\calA)(K)$ and $y\in\operatorname{Rep}_{n_\beta}(\calB)(K)$ . We set $z:=\varphi_2^*(x)$ .\par%
By Example \ref{bsp_dimvec}(a) we have $Z=\orbit{z}\cong \nicefrac{\GL{d,K}}{\stab{z}}$ and $Y_\beta=\orbit{y} \cong \nicefrac{\GL{d,K}}{\stab{y}}$ . Since $\orbit{z}(K) = \GL{d}(K).z$ and $\varphi_1^*$ restricts to a $\GL{d,K}$-equivariant map $Y_\beta\to Z$ , we may assume without loss of generality\footnote{For $\varphi_1^*(y)=g.z$ we may replace $y$ by $g\inv.y$ .} that $\varphi_1^*(y)=z=\varphi_2^*(x)$ . So $X_\alpha \times_Z Y_\beta(K) $ is non-empty.\par%
Furthermore taking fibres we obtain a commutative diagram
\begin{equation}
\label{yzx}
\begin{xy}\xymatrix{
X_\alpha \times_Z Y_\beta \ar[rr]^{\overline{\iota_\calB}} \ar[dd] && \nicefrac{\GL{d,K}}{\stab[{\GL{a,K}}]{y}} \ar@{-->}[dd]
\\
& \overline{\iota_\calB}\inv (y) \ar[rr] \ar[lu] \ar[dd] && \spec{K}\ar[dd]^{\id } \ar[lu]
\\
{X_\alpha} \ar@{-->}[rr]^-{\overline{\varphi}} && \nicefrac{\GL{d,K}}{\stab[{\GL{a,K}}]{z}}
\\
& \overline{\varphi}\inv(z) \ar[rr] \ar[lu] && \spec{K} \ar@{-->}[lu]
}\end{xy}
\end{equation}
where $\overline{\varphi}$ and $\overline{\iota_\calB}$ are given as the restrictions of $\varphi_2^*$ and $\iota_\calB^*$ . The bottom, top and back squares of \eqref{yzx} are pullback squares, hence, the front square is too and we obtain an isomorphism
\begin{equation*}
\overline{\iota_\calB}\inv (y) \cong \overline{\varphi}\inv(z)
\end{equation*}
Applying Lemma \ref{lem_slodowy_trick} we obtain isomorphisms
\begin{equation*}
X_\alpha \cong \GL{d,K} \times^{\stab{z}} \overline{\varphi}\inv(z) \quad , \quad X_\alpha \times_Z Y_\beta \cong \GL{d,K} \times^{\stab{y}} \overline{\iota_\calB}\inv (y)
\end{equation*}
So since $X_\alpha$ is connected by assumption, $\overline{\iota_\calB}\inv (y) \cong \overline{\varphi}\inv(z)$ and $ X_\alpha \times_Z Y_\beta$ are \mbox{connected} too.\\
About b): Since $\operatorname{Rep}_d(K[t,t\inv]) \cong \GL{d,K}$ and $\operatorname{Rep}_d(K)\cong \spec{K}$ are \mbox{connected,} we have isomorphisms $\calT(K[t,t\inv])\cong \bbN_0\cong \calT(K)$ given by $\mongrad$ respectively. So $\calT(\iota_\calA)$ is an isomorphism by part a).\par%
%
%
Now if $(x,g)\in \operatorname{Rep}_d(\calA)(K) \times \GL{d}(K) = \operatorname{Rep}_d(\calA *_K K[t,t\inv])(K)$ is a $K$-valued point, then $\Phi^*(x,g) = g\inv. (x,g) \in \orbit{(x,g)}(K)$ . So $(x,g)$ and $\Phi^*(x,g)$ lie in the same connected component, because $\orbit{(x,g)}$ is connected. This proves $\calT(\Phi)= \id$ .\\
About c): Denote the natural projection $\calA *_K K[t,t\inv]\to \HNN{\calA}{\calC}{\varphi_2,\varphi_3}$ by $\pi$ . By construction of the HNN extension we have $\pi\circ\iota_\calA\circ \varphi_2 = \pi\circ\Phi\circ\iota_\calA\circ \varphi_3$ . So using part b) we obtain that $\theta:= \calT(\iota_\calA)\circ\calT(\pi) : \calT(\HNN{\calA}{\calC}{\varphi_2,\varphi_3})\to \calT(\calA)$ factorizes over
\[
\operatorname{Eq}(\calT(\varphi_2),\calT(\varphi_3)) = \{m\in\calT(\calA) \mid \calT(\varphi_2)(m)=\calT(\varphi_3)(m)\} \subseteq \calT(\calA)
\]
We have to show that $\theta$ is an isomorphism.\par%
Denote the connected components of $\operatorname{Rep}_d(\calA)$ by $X_0,\ldots, X_a$ . As for a) we obtain a natural isomorphism $\operatorname{Rep}_d(\HNN{\calA}{\calC}{\varphi_2,\varphi_3}) \cong \operatorname{Eq}((\iota_\calA\circ\varphi_2)^*,(\Phi\circ\iota_\calA\circ\varphi_3)^*)$ and a decomposition
\[
\operatorname{Rep}_d(\HNN{\calA}{\calC}{\varphi_2,\varphi_3}) = \bigsqcup_{0\leq \alpha\leq a} (\iota_\calA^*\circ \pi^*)\inv(X_\alpha)
\]
into open and closed subsets $(\iota_\calA^*\circ \pi^*)\inv(X_\alpha) = (\pi^*)\inv ( X_\alpha \times_K \GL{d,K})$ and it remains to show
that $(\pi^*)\inv ( X_\alpha \times_K \GL{d,K})$ is connected and contains a rational point if $X_\alpha = \operatorname{Rep}_{m_\alpha}(\calA)$ corresponds to a dimension vector $m_\alpha\in \operatorname{Eq}(\calT(\varphi_2),\calT(\varphi_3))$ .\par%
We first check using the universal property of the equalizer $\operatorname{Rep}_d(\HNN{\calA}{\calC}{\varphi_2,\varphi_3})$ that a $K$-valued point \mbox{$(x,g)\in X_\alpha(K)\times \GL{d}(K)=(X_\alpha\times_K \GL{d,K})(K)$} lies in the image of the closed embedding $\pi^*$ if and only if
\begin{equation}
\label{equ_equal_dimvecpt}
\rho_x\circ\varphi_2 = (\iota_\calA \circ \varphi_2)^*(x,g) = (\Phi\circ\iota_\calA\circ\varphi_3)^*(x,g) 
= g\inv . (\rho_x\circ \varphi_3)
\end{equation}
The rational points associated to $\rho_x\circ \varphi_2$ and $\rho_x\circ \varphi_3$ belong to the same connected component $Z\subseteq \operatorname{Rep}_d(\calC)$ , because we assumed $m_\alpha\in \operatorname{Eq}(\calT(\varphi_2),\calT(\varphi_3))$ . Again using Example \ref{bsp_dimvec}(a) we know that $Z=\orbit{z}\cong \nicefrac{\GL{d,K}}{\stab{z}}$ for some \mbox{$z\in \operatorname{Rep}_d(\calC)(K)$ .} Hence, there is a $g\in\GL{d}(K)$ satisfying $\rho_x\circ \varphi_2 = g\inv . (\rho_x\circ \varphi_3)$ which yields a rational point in $(\pi^*)\inv ( X_\alpha \times_K \GL{d,K})$ .\par%
Now denote the restriction of \mbox{$(\pi\circ\iota_\calA\circ \varphi_2)^*$} to $(\pi^*)\inv ( X_\alpha \times_K \GL{d,K})$ by $\psi$ . The criterion \eqref{equ_equal_dimvecpt} yields that \mbox{${\psi}\inv(z)\cong (\varphi_2^*)\inv(z) \times_K \stab{z}$ .} So as $\stab{z}$ is geometrically irreducible by Example \ref{bsp_special_stabilizer}, ${\psi}\inv(z)$ is connected if and only if $(\varphi_2^*)\inv(z)$ is.\footnote{see \cite[Tag 0385]{xxstacks}} We now again use Lemma \ref{lem_slodowy_trick} to obtain isomorphisms
\begin{equation*}
X_\alpha \cong \GL{d,K} \times^{\stab{z}} (\varphi_2^*)\inv(z) \quad , \quad 
(\pi^*)\inv ( X_\alpha \times_K \GL{d,K}) \cong \GL{d,K} \times^{\stab{z}} {\psi}\inv(z)
\end{equation*}
So $(\pi^*)\inv ( X_\alpha \times_K \GL{d,K})$ is connected, because the connected component $X_\alpha$ is.\par%
About (d): The claim follows from the decomposition \eqref{equ_decomp_groupalg} of $K[\calG]$ and Example \ref{bsp_dimvec}(a) by repeatedly applying part (a) and (c) above.
\end{proof}
As a corollary from \eqref{equ_dimvec_total} we obtain that the dimension vector monoid $\calT(K[\calG])$ does not depend on the choice of the suitable field $K$ : First let $\calH$ be a finite group and $\calF\subseteq \calH$ be a subgroup. Using well-known arguments from the representation theory of finite groups\footnote{see \cite[§14.6, §15.1 \& Prop. 43 in §15.5]{xxserreReprFin}
} and Example \ref{bsp_dimvec}(a) one first shows that the monoids $\calT(K[\calF])$ and $\calT(K[\calH])$ as well as the homomorphism $\calT(K[\calH])\to \calT(K[\calF])$ do not depend on $K$ . So since $\calT(K[\calG])$ is the limit of a diagram of monoids which itself does not depend on $K$ , $\calT(K[\calG])$ does not depend on $K$ as well. We will therefore drop $K$ from the notation and simply write $\calT(\calG)$ .
\par%
We conclude our current discussion of dimension vectors with a few general \mbox{remarks.} However, the readers may feel free to skip forward to Section \ref{section_examples} for some \mbox{hands-on} \mbox{examples} at this point. We first note another immediate consequence of the \mbox{isomorphism} \eqref{equ_dimvec_total}: $\calT(\calG)$ is equipped with a canonical embedding into the free commutative monoid $\prod_i \calT(\calG_i)$ . This does not imply that $\calT(\calG)$ is free itself, but is a huge restriction on the class of monoids arising as $\calT(\calG)$ .\par%
Moreover $\calT(\calG)$ comes with a canonical homomorphism $\calT(\calG) \to \calT(\calG_i)$ for each $0\leq i\leq I$ and a canonical homomorphism $\calT(\calG) \to \calT(\calG_j')$ for each \mbox{$1\leq j\leq I+J$ .} We say that $m\in\calT(\calG)$ \defit{lies over} $m_i\in\calT(\calG_i)$ for $0\leq i\leq I$ and $u_j\in\calT(\calG_j')$ for $1\leq j\leq I+J$ if these are the images of $m$ under these canonical homomorphisms. Note that these images uniquely determine $m$ due to the isomorphism \eqref{equ_dimvec_total}.\par%
For $c\in \bbNe$ and $m\in\calT(\calG)$ we write $c\vert m$ if there is an $n\in\calT(\calG)$ fulfilling $m = c.n = n+\ldots +n$ . Such an $n$ is necessarily unique and we denote it by \mbox{$\nicefrac{m}{c}:=n$ .} Moreover $\{c\in\bbNe\mid c\vert m\}$ is a finite set $-$ this as well as the uniqueness of $\nicefrac{m}{c}$ are immediate consequences of the embedding $\calT(\calG)\hookrightarrow \prod_i \calT(\calG_i)$ . We denote
\begin{equation}
\label{equ_def_gcddimvec}
\operatorname{gcd}(m) := \max\{c\in\bbNe \mid c\vert m\} = \operatorname{lcm}\{c\in\bbNe \mid c\vert m\}
\end{equation}
Another important property of dimension vectors is that they are additive on short exact sequences, i.e. for every short sequence $0\to \calN\to \calW\to \calM\to 0$ of left modules over a finite type $K$-algebra $\calA$ we have $\dimvec(\calW)=\dimvec(\calN)+\dimvec(\calM)$ . If the sequence splits, this is true by definition of $\calT(\calA)$ . For the general case one uses {iv) and (v) in Subsection \ref{subsec_geometric_methods}.}
\subsection{Homological Euler form}
\label{subsec_euler_form}
We now want to discuss another object which again has a well-known analogue in quiver representation theory: the \defit{\mbox{(homological)} Euler form.} For $\calA$ a (left hereditary) $K$-algebra and finite dimensional left \mbox{$\calA$-modules} $\calM,\calN$ we define
\[
\langle\calM,\calN\rangle_\calA := \dimension[K]{ \Hom[\calA]{\calM,\calN}} - \dimension[K]{ \Ext[1]{\calA}{\calM,\calN}}
\]
\begin{bsp}\thmabsatz
\label{bsp_eulerform}
Let $\calC$ be a completely split finite dimensional semisimple \mbox{$K$-algebra.} Recall that $\calC$ may be written as \eqref{equ_normalform_semisimple}. $\calC$ admits precisely $c:=c_1+\ldots + c_e$ pairwise \mbox{non-isomorphic} (absolutely) simple modules $-$ choose a representative $\calL_\gamma$ for each \mbox{isomorphism} class. For two arbitrary finite dimensional left $\calC$-modules
\[
\calM = \bigoplus_{\gamma=0}^{c-1} {\calL_\gamma}^{\oplus m(\gamma)} \quad , \quad \calN = \bigoplus_{\gamma=0}^{c-1} {\calL_\gamma}^{\oplus n(\gamma)}
\]
we compute the homological Euler form
\[
\langle\calM,\calN\rangle_\calC = \dimension[K]{\Hom[\calC]{\calM,\calN}} = \sum_{\gamma=0}^{c-1} m(\gamma) n(\gamma)
\]
by using Schur's Lemma for the absolutely simple modules $\calL_\gamma$ . For \mbox{$m:=\dimvec(\calM)$} and $n:=\dimvec(\calN)$ we also introduce the notion $\langle m,n\rangle_\calC:= \langle\calM,\calN\rangle_\calC$ which is \mbox{well-defined,} since $\dimvec: \iso[d](\calC)\to \calT_d(\calC)$ is bijective by Example \ref{bsp_dimvec}(a).
\end{bsp}
As for the dimension vector monoid we now want to compute the homological Euler form of the group algebra of a finitely generated virtually free group $\calG$ over a suitable field.
\begin{sa}\thmabsatz
\label{sa_Eulerform_virt_free}
Let $\calA$ and $\calB$ be left hereditary finite type $K$-algebras, $\calC$ a finite dimensional semisimple $K$-algebra and $\varphi_1:\calC\to \calB$ , $\varphi_2,\varphi_3:\calC\to \calA$ $K$-algebra homomorphisms.
\startabc
\item Consider the pushout given by $\varphi_1,\varphi_2$ and let $\calM,\calN$ be finite dimensional left $\calA*_\calC \calB$-modules. The homological Euler form of $\calA*_\calC \calB$ is given by
\begin{equation*}
\langle\calM,\calN\rangle_{\calA*_\calC \calB} = \langle\calM,\calN\rangle_{\calA} + \langle\calM,\calN\rangle_{\calB} - \langle\calM,\calN\rangle_{\calC}
\end{equation*}
\item Consider the HNN extension $\HNN{\calA}{\calC}{\varphi_2,\varphi_3}$ and let $\calM,\calN$ be finite dimensional left $\HNN{\calA}{\calC}{\varphi_2,\varphi_3}$-modules. The homological Euler form of $\HNN{\calA}{\calC}{\varphi_2,\varphi_3}$ is given by
\begin{equation*}
\langle\calM,\calN\rangle_{\HNN{\calA}{\calC}{\varphi_2,\varphi_3}} = \langle\calM,\calN\rangle_{\calA}  - \langle\calM,\calN\rangle_{\calC}
\end{equation*}
\item If $\calG$ is the finitely generated virtually free group given by \eqref{equ_virt_free_decomp} and $K$ is suitable for $\calG$ , then $\langle -,- \rangle_{K[\calG]}$ is given by
\begin{equation}
\label{equ_eulerform_total}
\langle \calM,\calN\rangle_{K[\calG]} = \sum_{i=0}^I \langle m_i, n_i\rangle_{K[\calG_i]} - \sum_{j=1}^{I+J} \langle u_j , v_j\rangle_{K[\calG_j']}
\end{equation}
where $\dimvec(\calM)$ is the dimension vector lying over $m_i\in\calT(\calG_i)$ for $0\leq i\leq I$ and $u_j\in\calT(\calG_j')$ for $1\leq j\leq I+J$ and $\dimvec(\calN)$ is lying over $n_i\in\calT(\calG_i)$ and $v_j\in\calT(\calG_j')$ respectively.
\stopabc
\end{sa}
\begin{proof}
Let $\calD$ be a finite type $K$-algebra and $\calW$ a $K$-linear $(\calD,\calD)$-bimodule.\footnote{i.e. a left $\calD\otimes_K \calD\op$-module} We consider the $K$-linear map $\eta:\calW\to \Der[K]{\calD,\calW}$ which sends $w\in\calW$ to its inner derivation $\eta(w)=(f\mapsto f\lsc w - w\rsc f)$ and obtain an exact sequence
\begin{equation}
\label{equ_exact_sequ_Euler}
0\to \Kern{\eta}\to \calW\overset{\eta}{\to} \Der[K]{\calD,\calW} \to \Kokern{\eta}\to 0
\end{equation}
For the bimodule $\calW=\Hom[K]{\calM,\calN}$ we obtain $\Kern{\eta}=\Hom[\calD]{\calM,\calN}$ and \mbox{$\Kokern{\eta} \cong \Ext[1]{\calD}{\calM,\calN}$ .}\footnote{see \cite[Lemma 9.1.9 \& Lemma 9.2.1]{xxweibel}} Hence, \eqref{equ_exact_sequ_Euler} yields
\[
\langle\calM,\calN\rangle_\calD = \dimension[K]{\calM}\cdot\dimension[K]{\calN} - \dim_K \Der[K]{\calD,\Hom[K]{\calM,\calN}}
\]
So we can reformulate the claimed identity a) as
\[
\dim_K {\Der[K]{\calD,\calW}} = \dim_K {\Der[K]{\calA,\calW}} + \dim_K {\Der[K]{\calB,\calW}} - \dim_K {\Der[K]{\calC,\calW}}
\]
for $\calD= {\calA*_\calC \calB}$ and $\calW=\Hom[K]{\calM,\calN}$ and the claimed identity b) takes the form
\[
\dim_K {\Der[K]{\calD,\calW}} = \dim_K {\Der[K]{\calA,\calW}} + \dimension[K]{\calW}  - \dim_K {\Der[K]{\calC,\calW}}
\]
for $\calD=\HNN{\calA}{\calC}{\varphi_2,\varphi_3}$ and $\calW=\Hom[K]{\calM,\calN}$ .\par%
For the identity a) we use that $\Der[K]{\calD,\calW}$ is the $K$-vector space pullback \mbox{induced} by $\varphi_1^*$ and $\varphi_2^*$ , hence, $\Der[K]{\calD,\calW}$ is the kernel of the map
\[
({\varphi_2}^*, - {\varphi_1}^*): \Der[K]{\calA,\calW}\oplus \Der[K]{\calB,\calW} \to \Der[K]{\calC,\calW}
\]
which is surjective, because $\calC$ is separable, i.e. every derivation of $\calC$ is inner.\footnote{See \cite[Prop. 4.2]{xxcuntzquillen}. Note that we used here that $K$ is assumed to be perfect. This could be avoided by assuming that $\calC$ is completely split or more generally separable.}\par%
The identity b) is proven similarly:
\[
\Der[K]{\calD,\calW} \overset{\pi^*}{\to} \Der[K]{\calA *_K K[t,t\inv ],\calW} 
\underset{(\varphi_2')^*}{\overset{(\varphi_3')^*}{\rightrightarrows}}
\Der[K]{\calC,\calW}
\]
is an equalizer diagram of vector spaces, i.e. $\Der[K]{\calD,\calW}$ is the kernel of the map 
\[
\Der[K]{\calA *_K K[t,t\inv ],\calW}  \overset{(\varphi_2')^* - (\varphi_3')^*}{\to} \Der[K]{\calC,\calW}
\]
which is surjective as well, because $\calC$ is a separable $K$-algebra. This proves b), since we have an isomorphism
\[
\Der[K]{\calA *_K K[t,t\inv ],\calW} \cong \Der[K]{\calA ,\calW} \oplus \calW \quad , \quad \delta\mapsto ( \delta\circ\iota_\calA ,\delta(t)) 
\]
The proof of c) is now just a repeated application of the parts a) and b).
\end{proof}
%
%
%
Since the righthand side of the formula \eqref{equ_eulerform_total} only depends on the dimension vectors $\dimvec(\calM)$ and $\dimvec(\calN)$ and is $\bbN_0$-linear in both arguments, Proposition \ref{sa_Eulerform_virt_free} yields that the homological Euler form induces a well-defined $\bbN_0$-bilinear map
\[
\langle-,-\rangle_{K[\calG]} : \calT(\calG)\times \calT(\calG) \to \bbZ
\]
Furthermore we see from formula \eqref{equ_eulerform_total} that $\langle-,-\rangle_{K[\calG]}$ does not depend on $K$ . So we will simply denote it by $\langle -,- \rangle_\calG$ . Moreover \eqref{equ_eulerform_total} combined with Example \ref{bsp_eulerform} shows that $\langle -,-\rangle_\calG$ is symmetric.\par%
As before we postpone explicit examples to Section \ref{section_examples}, but the readers may feel free to skip forward to it now.
\subsection{Counting representation spaces}
\label{subsec_counting_representation_spaces}
Now assume $K=\bbF_q$ is finite. We want to show in this subsection that the connected components $\operatorname{Rep}_m(\bbF_q[\calG])$ , $m\in\calT(\calG)$ are polynomial count if $\bbF_q$ is suitable for $\calG$ . As before we start with the case of semisimple algebras.
\begin{bsp}\thmabsatz
\label{bsp_counting_repspace}
Let $\calC$ be a completely split finite dimensional semisimple \mbox{$\bbF_q$-algebra} with \mbox{(absolutely)} simple left modules $\calL_0, \ldots , \calL_{c-1}$ . If $\calM$ is a finite dimensional left $\calC$-module of \mbox{dimension} vector $m=\dimvec(\calM)=\sum_{\gamma} m(\gamma).\dimvec(\calL_\gamma)$ we know from Example \ref{bsp_dimvec}(a) that $\operatorname{Rep}_m(\calC) =\orbit{\calM} \cong\nicefrac{\GL{\vert m\vert,K}}{\stab{x_\calM}}$ . Since $\stab{x_\calM}\cong \prod_{\alpha} \GL{m(\alpha),\bbF_q}$ is special (see Example \ref{bsp_special_stabilizer}), we obtain
\[
\# \operatorname{Rep}_m(\calC)(\bbF_{q^\alpha}) = \#\nicefrac{\GL{\vert m\vert}(\bbF_{q^\alpha})}{\stab{x_\calM}(\bbF_{q^\alpha})} = \frac{\countgl[\vert m\vert]}{\prod_{\gamma=0}^{c-1} \countgl[m(\gamma)]} (q^\alpha)
\]
Using {(i) in Subsection \ref{subsec_geometric_methods}} we see that the rational function $P_m^\calC := \nicefrac{\countgl[\vert m\vert]}{\prod_{\gamma=0}^{c-1} \countgl[m(\gamma)]}$ is in fact a counting polynomial for $\operatorname{Rep}_m(\calC)$ .
\end{bsp}
Similar to $\calT(\calG)$ and $\langle -,-\rangle_\calG$ we give a full description of the counting polynomials of $\operatorname{Rep}_m(\bbF_q[\calG])$ .
\begin{sa}\thmabsatz
\label{sa_counting_repspace}
Let $\calA$ be a finite type $\bbF_q$-algebra, $\calB$ and $\calC$ completely split finite dimensional semisimple $\bbF_q$-algebras and $\varphi_1:\calC\to \calB$ , $\varphi_2,\varphi_3:\calC\to \calA$ \mbox{homomorphisms} of \mbox{$\bbF_q$-algebras.} For $d\in\bbN_0$ fix dimension vectors $m\in\calT_d(\calA)$ , $n\in\calT_d(\calB)$ and \mbox{$u\in\calT_d(\calC)$ .}
\startabc
\item Consider the pushout $\calA *_\calC \calB$ given by $\varphi_1,\varphi_2$ and assume that $(m,n)$ is a dimension vector in $\calT(\calA) \times_{\calT(\calC)} \calT(\calB) =\calT(\calA *_\calC \calB)$ lying over $u$ . If $\operatorname{Rep}_m(\calA)$ admits a counting polynomial $P_m^\calA$ , then the rational function $P_{(m,n)}^{\calA *_\calC \calB} := \nicefrac{P_m^\calA P_n^\calB}{P_u^\calC}$
is a counting polynomial for $\operatorname{Rep}_{(m,n)}(\calA *_\calC \calB)$ .
\item Consider the HNN extension $\HNN{\calA}{\calC}{\varphi_2,\varphi_3}$ and assume that $m$ is an element of the equalizer $\operatorname{Eq}(\calT(\varphi_2),\calT(\varphi_3)) = \calT(\HNN{\calA}{\calC}{\varphi_2,\varphi_3})$ lying over $u$ . If $\operatorname{Rep}_m(\calA)$ \mbox{admits} a counting polynomial $P_m^\calA$ , then the rational function $P_{m}^{\HNN{\calA}{\calC}{\varphi_2,\varphi_3}} := \nicefrac{P_m^\calA \countgl[d]}{P_u^\calC}$
is a counting polynomial for $\operatorname{Rep}_m(\HNN{\calA}{\calC}{\varphi_2,\varphi_3})$ .
\item If $\calG$ is the finitely generated virtually free group given by \eqref{equ_virt_free_decomp} and $\bbF_q$ is suitable for $\calG$ , then
\begin{equation}
\label{equ_counting_repspace_formula}
P_m^\calG 
:= 
\countgl[d]^J \frac{\prod_{i=0}^I P_{m_i}^{\calG_\alpha}}{\prod_{j=1}^{I+J} P_{u_j}^{\calG_j'}}
=
\countgl[d]^J \frac{\prod_{j=1}^{I+J}\prod_{\gamma}^{} \countgl[u_j(\gamma)]}{\prod_{i=0}^I \prod_{\beta}^{} \countgl[m_i(\beta)]}
\end{equation}
is a counting polynomial for $\operatorname{Rep}_m(\bbF_q[\calG])$ where $m\in\calT(\calG)$ is the dimension vector lying over $m_i\in \calT(\calG_i)$ for $0\leq i\leq I$ and over $u_j\in \calT(\calG_j')$ for \mbox{$1\leq j\leq I+J$ .}
\stopabc
\end{sa}
\begin{proof}
About (a): As in the proof of Proposition \ref{sa_dimvec_virt_free}(a) we may express $\operatorname{Rep}_m(\calA)$ and $\operatorname{Rep}_{(m,n)}(\calA *_\calC \calB)$ as associated fibre spaces
\[
\operatorname{Rep}_m(\calA) \cong \GL{d,\bbF_q} \times^{\stab{z}} Y \quad , \quad \operatorname{Rep}_{(m,n)}(\calA *_\calC \calB) \cong \GL{d,\bbF_q} \times^{\stab{y}} Y
\]
for $\stab{y}\subseteq \stab{z}\subseteq \GL{d,K}$ the stabilizers of points $y\in\operatorname{Rep}_n(\calB)(\bbF_q)$ , $z\in \operatorname{Rep}_u(\calC)(\bbF_q)$ and $Y$ an affine finite type $\bbF_q$-scheme with $\stab{z}$-action. Since $\stab{y}$ and $\stab{z}$ are special, we may use \eqref{equ_special_bijection_rational_points} to obtain
\[
\# \operatorname{Rep}_m(\calA *_\calC \calB)(\bbF_{q^\alpha}) = \frac{\countgl[d]({q^\alpha})}{\# \stab{y}(\bbF_{q^\alpha})} \#Y(\bbF_{q^\alpha}) = \frac{\countgl[d]({q^\alpha})}{\# \stab{y}(\bbF_{q^\alpha})} \frac{\# \stab{z}(\bbF_{q^\alpha})}{\countgl[d]({q^\alpha})} P_m^\calA(q^\alpha) 
\]
Using Example \ref{bsp_counting_repspace} this proves part (a).
\par%
About (b): As above we use the proof of Proposition \ref{sa_dimvec_virt_free}(c) to obtain
\[
\operatorname{Rep}_m(\calA) \cong \GL{d,\bbF_q} \times^{\stab{z}} Y \quad , \quad \operatorname{Rep}_m(\HNN{\calA}{\calC}{\varphi_2,\varphi_3}) \cong \GL{d,\bbF_q} \times^{\stab{z}} \left( Y \times_{\bbF_q} \stab{z}\right)
\]
for $\stab{z}\subseteq \GL{d,\bbF_q}$ the stabilizer of a point $z\in\operatorname{Rep}_u(\calC)(\bbF_q)$ and we calculate
\[
\#\operatorname{Rep}_m(\HNN{\calA}{\calC}{\varphi_2,\varphi_3})(\bbF_{q^\alpha}) = \countgl[d](q^\alpha) \#Y(\bbF_{q^\alpha}) = \countgl[d](q^\alpha) \frac{\# \stab{z}(\bbF_{q^\alpha})}{\countgl[d]({q^\alpha})} P_m^\calA(q^\alpha)
\]
About (c): We obtain $P_m^\calG$ by repeatedly applying part (a) and (b) above to our decomposition \eqref{equ_decomp_groupalg} of $K[\calG]$ (note that $J$ is the number of HNN-extensions involved in \eqref{equ_decomp_groupalg}). The second expression comes from Example \ref{bsp_counting_repspace} by cancelling out the $\countgl[d]$ occuring in the numerator and denominator of the fraction. 
\end{proof}
The formula \eqref{equ_counting_repspace_formula} in particular shows that the polynomials $P_m^\calG$ are independent of the choice of a finite suitable field for $\calG$ .
%
%
%
\section{Hall algebra methods}
\label{section_4_hall_algebra_methods}
\fakesubsection{Generating functions}
\noindent Consider the field $\bbQ(s)$ of rational functions in the variable $s$ as well as its subring
\[
\localpol = \left\{ \nicefrac{P}{Q}\in \bbQ(s) \mid Q(q)\neq 0\right\}
\]
The $\bbQ$-algebra homomorphisms $\bbQ(s) \hookleftarrow \localpol \overset{\ev}{\twoheadrightarrow} \bbQ$ induce homomorphisms of $\calT(\calG)$-graded \mbox{$\bbQ$-algebras}
\begin{equation}
\label{equ_tricho_graded}
\bbQ(s)[\calT(\calG)] \hookleftarrow \localpol[\calT(\calG)] \overset{\ev}{\twoheadrightarrow} \bbQ [\calT(\calG)]
\end{equation}
The homomorphism $\mongrad: \calT(\calG)\to \bbN_0$ endows every $\calT(\calG)$-graded algebra $\calC$ with an $\bbN_0$-grading where $\calC_d$ is spanned by all homogeneous elements with degree in \mbox{$\calT_d(\calG)$ .} So we may endow $\calC$ as every $\bbN_0$-graded algebra with a topology by taking the ideals $\bigoplus_{\delta\geq d} \calC_\delta$ as a neighbourhood basis of $0$ and form the completion $\widehat{\calC} = \prod_{\delta\geq 0} \calC_\delta$ of $\calC$ with respect to it. We have the following facts on completions of (graded) algebras:
\begin{compactenum}[(i)]
{\item[(vii)] An element $(f_\delta)_{\delta\geq 0}\in \widehat{\calC}$ is a multiplicative unit if and only if $f_0\in\calC_0$ is a unit.}
\item[(viii)] Every graded homomorphism $\calC\to \calC'$ between $\calT(\calG)$-graded algebras extends uniquely to a continuous algebra homomorphism $\widehat{\calC}\to \widehat{\calC'}$ .
\end{compactenum}
So by taking completions of \eqref{equ_tricho_graded} we obtain continuous \mbox{$\bbQ$-algebra} homomorphisms
\begin{equation}
\label{equ_tricho_continuous}
\cplmonoidalg[\bbQ(s)]{\calT(\calG)} \hookleftarrow \cplmonoidalg[\localpol]{\calT(\calG)} \overset{\ev}{\twoheadrightarrow} \cplmonoidalg{\calT(\calG)}
\end{equation}
We now define a second multiplication on the monoid algebras considered above: The so called \defit{twisted multiplication} on $\bbQ [\calT(\calG)]$ is given by bilinear extension of
\[
t^m * t^n:= q^{-\langle m,n\rangle_\calG} . t^{m+n}
\]
Analogously we define $t^m * t^n:= s^{-\langle m,n\rangle_\calG} . t^{m+n}$ on $\bbQ(s)[\calT(\calG)] $ and \mbox{$ \localpol[\calT(\calG)]$ .} We denote the resulting $\calT(\calG)$-graded $\bbQ$-algebras by $\twmonoidalg{$q$-}{\calT(\calG)}$ , $\twmonoidalg[\bbQ(s)]{}{\calT(\calG)}$ and $\twmonoidalg[\localpol]{}{\calT(\calG)}$ . As for the monoid algebras we have $\calT(\calG)$-graded $\bbQ$-algebra homomorphisms analogous to \eqref{equ_tricho_graded} and continuous $\bbQ$-algebra homomorphisms like \eqref{equ_tricho_continuous} between their twisted versions.\par%
The twisted monoid algebras are in fact isomorphic to their untwisted \mbox{counterparts.} To construct explicit isomorphisms between them we need a monoid homomorphism \mbox{$\correctionform: \calT(\calG)\to \bbZ$} which satisfies
\[
\langle m,m\rangle_\calG \equiv \correctionform(m) \mod (2) \quad \fa m\in\calT(\calG)
\]
We construct a distinguished $\correctionform$ to show existence, but everything that follows does not depend on this choice. For a finite group $\calF$ we have an identification $\calT(\calF)\cong {\bbN_0}^c$ and may take $\correctionform(m):=\sum_{\gamma=0}^{c-1} m(\gamma)$ . For the general case of $\calG$ we can mimic our computation of the Euler form and define
\begin{equation*}
\correctionform(m) := \sum_{i = 0}^I \sum_\gamma m_i(\gamma) - \sum_{j=1}^{I+J}\sum_\gamma u_j(\gamma)
\end{equation*}
where $m\in\calT(\calG)$ is the dimension vector lying over $m_i\in {\bbN_0}^{c_i}\cong \calT(\calG_i)$ and over $u_j\in{\bbN_0}^{c_j'}\cong \calT(\calG_j')$ . We may now define a $\bbQ$-vector space isomorphism
\[
\calS : \twmonoidalg{$q$-}{\calT(\calG)} \to \bbQ[\calT(\calG)] \quad , \quad \calS(t^m):= q^{\frac{1}{2}(\langle m,m\rangle_\calG-\correctionform(m))} t^m
\]
and isomorphisms $\twmonoidalg[\bbQ(s)]{}{\calT(\calG)}\cong \bbQ(s)[\calT(\calG)]$ , $\twmonoidalg[\localpol]{}{\calT(\calG)} \cong \localpol [\calT(\calG)]$ via $\calS(t^m) :=s^{\frac{1}{2}(\langle m,m\rangle_\calG-\correctionform(m))} t^m$ . We call each of the maps $\calS$ \defit{shift operator.} By construction the shift operators preserve the $\calT(\calG)$-grading and using that $\langle -,-\rangle_\calG$ is symmetric one can deduce that they are isomorphisms of graded algebras. Hence, they extend uniquely to continuous algebra homomorphisms between the completed monoid algebras.\par%
Shift operators like $\calS$ have already appeared in Mozgovoy-Reineke's treatment of the free group case in \cite{xxMozgovoyReineke2015}. To get rid of the \defit[2]{correction form} $\correctionform$ we could also define a shift operator by $\calS'(t^m) := q^{\frac{1}{2}\langle m,m\rangle_\calG} t^m$ . This would mean however that we have to work with $\bbQ[\sqrt{q}]$ instead of $\bbQ$ , $\bbQ[\sqrt{q}](\sqrt{s})$ instead of $\bbQ(s)$ etc.
\fakesubsection{The Hall algebra integral}
Now fix a finite field $K=\bbF_q$ which is suitable for $\calG$ . We briefly recall the construction of the \defit{finitary Hall algebra} $\bfH(\calA)$ of a finite type $\bbF_q$-algebra $\calA$ :\par%
Denote by $\iso(\calA):= \bigsqcup_{d\geq 0} \iso[d](\calA)$ the set of all isomorphism classes of finite dimensional left $\calA$-modules.\footnote{Analogously we denote by $\ssimp(\calA)$ , $\simp(\calA)$ and $\absimp(\calA)$ the sets of all isomorphism classes of semisimple, simple and absolutely simple modules respectively.} $\bfH(\calA)$ is defined as the free $\bbQ$-vector space on the basis $\iso(\calA)$ with multiplication given by $[\calM]\cdot [\calN] = \sum_{[\calW]} F_{\calM,\calN}^\calW [\calW]$ via structure coefficients
\[
F_{\calM,\calN}^\calW := \#\{ \calL\subseteq \calW \text{ left $\calA$-submodule} \mid \calL\cong \calN, \nicefrac{\calW}{\calL} \cong \calM \}
\]
Since dimension vectors are {additive on short exact sequences}, $\bfH(\calA)$ is $\calT(\calA)$-graded $-$ $\bfH_m(\calA)$ is the $\bbQ$-linear span of $\{[\calM]\in\iso(\calA)\mid \dimvec(\calM)=m\}$ . In particular the homomorphism $\mongrad: \calT(\calA)\to \bbN_0$ induces an $\bbN_0$-grading $\bfH(\calA) = \bigoplus_{\delta\geq 0} \bfH_\delta(\calA)$ . As for the monoid algebras \eqref{equ_tricho_graded} we may complete $\bfH(\calA)$ with respect to this $\bbN_0$-grading. Denote the completed finitary Hall algebra by $\cplhallalg{\calA}$ .\par%
We consider the element $\varepsilon:= \sum_{[\calM]\in\iso(\calA)} [\calM] \in\cplhallalg{\calA}$ which is a \mbox{multiplicative} unit by {(vii) from above}. It was shown by M. Reineke in \cite[Lemma 3.4]{xxreinekecounting} that the coefficients $e_\calM$ of the inverse \mbox{$\varepsilon\inv = \sum_{[\calM]} e_\calM [\calM]$} are given by
\begin{equation}
\label{equ_coeff_epsilon_inverse}
\begin{cases}
\prod\limits_{[\calL]\in\simp(\calA)} (-1)^{a_\calL} \#\End[\calA]{\calL}^{\nicefrac{a_\calL(a_\calL-1)}{2}} &, \quad \text{if } \calM= \bigoplus\limits_{[\calL]\in\simp(\calA)} \calL^{\oplus a_\calL} \text{ semisimple}
\\
\hfil 0 &, \quad \text{if } \calM \text{ not semisimple}
\end{cases}
\end{equation}
For $\calA= \bbF_q[\calG]$ we have a homomorphism of $\calT(\calG)$-graded $\bbQ$-algebras
\begin{equation}
\label{equ_Hall_algebra_integral}
\int: \bfH(\bbF_q[\calG]) \to \twmonoidalg{$q$-}{\calT(\calG)} \quad , \quad \int([\calM]) := \frac{1}{\#\Aut[{\bbF_q[\calG]}]{\calM}} t^{\dimvec(\calM)}
\end{equation}
analogously to \cite[Lemma 3.3]{xxreinekecounting}.\footnote{Reineke's proof of the analogous construction for quiver representations in \cite[§3]{xxreinekecounting} holds without any changes. It only relies on the fact that $\bbF_q[\calG]$ is (left) hereditary and would work for any left hereditary algebra for which the Euler form factorizes over a proper analogue of the dimension vector monoid $\calT(K\vec{Q})$ (like $\calT(\calG)$ in our case).} The map $\int$ is called a \defit{Hall algebra integral.} By {(viii) from above} it extends uniquely to a continuous $\bbQ$-algebra homomorphism between the completions .\par%
We summarize the situation with the following commutative diagram:
\begin{equation*}
\begin{xy}
\xymatrix{
\twcplmonoidalg[\bbQ(s)]{}{\calT(\calG)}{} \ar[d]^\calS_\cong &
\ar@{_{(}->}[l] \twcplmonoidalg[\localpol]{}{\calT(\calG)} \ar@{->>}[r]^{\ev} \ar[d]^\calS_\cong &
\twcplmonoidalg{q-}{\calT(\calG)} \ar[d]^\calS_\cong & \ar@{->>}[l]_\int \cplhallalg{\bbF_q[\calG]}
\\
\cplmonoidalg[\bbQ(s)]{\calT(\calG)} &
\ar@{_{(}->}[l] \cplmonoidalg[\localpol]{\calT(\calG)} \ar@{->>}[r]^{\ev} &
\cplmonoidalg{\calT(\calG)}
}\end{xy}
\end{equation*}
Most of the actual computations we are interested in happen in the ring $\cplmonoidalg[\bbQ(s)]{\calT(\calG)}$ while our knowledge of the representation theory of $\bbF_q[\calG]$ comes from the completed Hall algebra $\cplhallalg{\bbF_q[\calG]}$ . So the procedure is the following: First we observe an interesting identity in $\cplhallalg{\bbF_q[\calG]}$ , then we map it to $\cplmonoidalg{\calT(\calG)}$ and check whether we can reasonably\footnote{This mostly means that the lift should not depend on the specific prime power $q$ .} lift it along $\ev$ . Afterwards we can manipulate the obtained identity within $\cplmonoidalg[\bbQ(s)]{\calT(\calG)}$ .
\section{Counting polynomials}
\label{section_5_counting_polynomials}
\fakesubsection{Refined counting functions}
\noindent After introducing a lot of machinery we now come back to our original {objective} of counting functions and relate them to our machinery. Let $\bbF_q$ be a suitable finite field for $\calG$ . For each dimension vector $m\in\calT(\calG)$ define the \defit{refined counting functions}
\begin{equation}
\label{def_refined_countfct}
\begin{array}{ccl}
\countfct[absim]{}{m}(q^\alpha) &:=& \#\{[\calM]\in \absimp(\bbF_{q^\alpha}[\calG]) \mid \dimvec(\calM)=m\}
\\[2pt]
\countfct{}{m}(q^\alpha) &:=& \#\{[\calM]\in \simp(\bbF_{q^\alpha}[\calG]) \mid \dimvec(\calM)=m\}
\\[2pt]
\countfct[ss]{}{m}(q^\alpha) &:=& \#\{[\calM]\in \ssimp(\bbF_{q^\alpha}[\calG]) \mid \dimvec(\calM)=m\}
\end{array}
\end{equation}
The refined counting functions $\countfct[ss]{}{m}$ and $\countfct[absim]{}{m}$ again count the rational points of GIT moduli spaces: All connected components $\operatorname{Rep}_m(\bbF_q[\calG]) \subseteq \operatorname{Rep}_{\vert m \vert}(\bbF_q[\calG])$ are $\GL{\vert m\vert,\bbF_q}$-invariant, their GIT quotients $\modulirep{\bbF_q[\calG]}{m} := \GITquot{\operatorname{Rep}_m(\bbF_q[\calG])}{\GL{\vert m\vert ,\bbF_q}}$ are the connected components of $\modulirep{\bbF_q[\calG]}{\vert m\vert }$ . Moreover there is a $\GL{\vert m\vert,\bbF_q}$-invariant open subscheme $\operatorname{Rep}_m^{\text{absim}}(\bbF_q[\calG])\subseteq \operatorname{Rep}_m(\bbF_q[\calG])$ such that
\[
\modulirep[absim]{\bbF_q[\calG]}{m} := \GITquot{\operatorname{Rep}_m^{\text{absim}}(\bbF_q[\calG])}{\GL{\vert m\vert ,\bbF_q}} = \modulirep{\bbF_q[\calG]}{m} \cap \modulirep[absim]{\bbF_q[\calG]}{\vert m\vert}\]
The connected components of $\modulirep[absim]{\bbF_q[\calG]}{d}$ are given by those \mbox{$\modulirep[absim]{\bbF_q[\calG]}{m}$ ,} $m\in\calT_d(\calG)$ which are non-empty. These GIT moduli spaces satisfy
\begin{equation}
\label{equ_counting_refined_moduli_spaces}
\countfct[ss]{}{m} \left(q^\alpha\right) = \# \modulirep{\bbF_q[\calG]}{m}(\bbF_{q^\alpha})
\quad , \quad
\countfct[absim]{}{m} \left(q^\alpha\right) = \# \modulirep[absim]{\bbF_q[\calG]}{m}(\bbF_{q^\alpha})
\end{equation}
We can recover the original counting functions \eqref{equ_def_countfct_ordinary} from the refined ones via the formula \mbox{$\countfct[xyz]{}{d} = \sum_{\vert m\vert =d} \countfct[xyz]{}{m}$ .} Moreover we define
\[
\countfct{}{m,c}(q^\alpha) := \#\{[\calM]\in \simp(\bbF_{q^\alpha}[\calG]) \mid \dimvec(\calM)=m, \dimension[\bbF_{q^\alpha}]{\End[{\bbF_{q^\alpha}[\calG]}]{\calM}} = c \}
\]
Analogously to \cite[§4]{xxreinekecounting} we obtain the identities\footnote{The first identity holds just by the definition of absolutely simple modules, the second identity can be obtained from Galois descent and Möbius inversion.}
\begin{equation}
\label{equ_identities_Galois_und_Schur}
\countfct[absim]{}{m}(q^\alpha) = \countfct[sim]{}{m,1}(q^\alpha) \quad , \quad \countfct[sim]{}{m,c}(q^{\alpha}) = \begin{cases}
%
%
\frac{1}{c}\sum_{\gamma\vert c} \mu(\gamma) \countfct[absim]{}{\nicefrac{m}{c}}(q^{\nicefrac{\alpha c}{\gamma}})
&,\quad \text{if } c\vert m
\\
\hfil 0
&,\quad \text{else}
\end{cases}
\end{equation}
Here $\mu: \bbN\to \{-1,0,1\}$ denotes the (classical) Möbius function. Since $\calT(\calG)$ embeds into a free commutative monoid, $\cplmonoidalg[\calC]{\calT(\calG)}$ can be embedded into a formal power series ring $\calC [\![ t_1 ,\ldots , t_a ]\!]$ , i.e. we may interprete the elements of $\cplmonoidalg[\calC]{\calT(\calG)}$ as formal power series. Important examples are
\[
\countfct[xyz]{}{}(q^\alpha) := \sum_{m\in\calT(\calG)} \countfct[xyz]{}{m}(q^\alpha)t^m \in \cplmonoidalg{\calT(\calG)}
\]
where $\text{xyz}\in \{ \text{absim}, \text{sim} , \text{ss}\}$ . Our goal is to lift the power series $\countfct[xyz]{}{}(q^\alpha)$ reasonably along $\ev[q^\alpha]$ , the coefficients $\countpol[xyz]{}{m}$ of such a lift $\countpol[xyz]{}{}$ will be the counting polynomials we are aiming for.
\fakesubsection{Plethystic exponential and logarithm}
%
%
%
%
We now briefly recall the construction of plethystic exponentials and \mbox{logarithms.} First note that $\cplmonoidalg[\bbQ(s)]{\calT(\calG)}$ is a local ring with maximal ideal
\[
\frakm:=\Bigl\{ \sum_{m\in\calT(\calG)} f_m t^m \in \cplmonoidalg[\bbQ(s)]{\calT(\calG)} \mid f_0 = 0 {\Bigr\}}
\]
which is open. The subset $1+\frakm$ is open as well and a topological group with respect to multiplication. $(\frakm,+)$ and $(1+\frakm,\cdot)$ are isomorphic as topological groups, mutually inverse continuous isomorphisms are given by
\begin{equation*}
\frakm \adjunction{\log}{\exp} 1+\frakm \quad , \quad \exp(f):= \sum_{\alpha\geq 0} \frac{f^\alpha}{\alpha !} \quad , \quad \log(1+f) := \sum_{\beta\geq 1} \frac{(-1)^{\beta+1}}{\beta} f^\beta
\end{equation*}
Note that $\exp$ and $\log$ are equally well-defined for $\cplmonoidalg[\localpol]{\calT(\calG)}$ and $\cplmonoidalg{\calT(\calG)}$ and that they commute with the homomorphisms \eqref{equ_tricho_continuous}, e.g. $\exp \circ \ev (f) = \ev\circ \exp(f)$ for each $f\in \frakm\cap \cplmonoidalg[\localpol]{\calT(\calG)}$ .\par%
For each $a\in\bbNe$ we consider the \defit{Adams operation}
\[
\psi_a : \cplmonoidalg[\bbQ(s)]{\calT(\calG)} \to \cplmonoidalg[\bbQ(s)]{\calT(\calG)} \quad ,\quad \psi_a\Bigl(\sum_{m} f_m t^m {\Bigr)} := \sum_{m} f_m(s^a)t^{a.m}
\]
which is a continuous $\bbQ$-algebra homomorphism. They give rise to the mutually inverse continuous group automorphisms
\[
 \frakm \adjunction{\Psi\inv}{\Psi} \frakm \quad , \quad \Psi(f) := \sum_{\alpha\geq 1} \frac{\psi_\alpha(f)}{\alpha} \quad , \quad \Psi\inv (f) = \sum_{\beta\geq 1} \mu(\beta)\frac{\psi_\beta(f)}{\beta}
\]
The \defit{plethystic exponential} and \defit{plethystic logarithm} are defined by \mbox{$\Exp:=\exp\circ\Psi$} and \mbox{$\Log:=\Psi\inv\circ\log$ .} They are by definition mutually inverse continuous group isomorphisms, i.e. they in particular fulfill the usual functional equations
\[
\Exp(f+g)=\Exp(f)\Exp(g) \quad ,\quad \Log(fg) = \Log(f)+\Log(g)
\]
Moreover the same identities hold for convergent infinite sums and products. $\Exp$ and $\Log$ can alternatively be defined on $\cplmonoidalg[\bbQ(\!(s)\!)]{\calT(\calG)}$ where $\bbQ(\!(s)\!)$ denotes the field of formal Laurent series. By some calculations in $\cplmonoidalg[\bbQ(\!(s)\!)]{\calT(\calG)}$ one can prove\footnote{See e.g. (2) in \cite[§2.3]{xxMozgovoyMcKay} for the calculation in the case $c=1$ .}
\begin{equation}
\label{equ_Heim}
\Exp\left(\frac{1}{1-s^c}t^m\right) = \left(\prod_{\alpha\geq 0} 1-s^{c\alpha}.t^m \right) \inv = \sum_{b\geq 0} \left( \prod_{\beta =1}^b (1-s^{c\beta}) \right)\inv . t^{b.m}
\end{equation}
Using the theorem of Krull-Remak-Schmidt for a product factorization of the power series $\countfct[ss]{}{}(q)$ and the second formula in \eqref{equ_identities_Galois_und_Schur} one can prove the following lemma.
\begin{lem}\thmabsatz
\label{lem_Exp(absim)_ss_kleinr}
The power series
\[
E(q) := \sum_{\substack{m\in\calT(\calG),\\ \beta\geq 1}} \frac{1}{\beta} \countfct[absim]{}{m}(q^\beta) t^{\beta.m}
\]
is for each $q$ convergent in $\cplmonoidalg[\bbQ]{\calT(\calG)}$ and satisfies $\exp(E(q))=\countfct[ss]{}{}(q)$ .
\end{lem}
See \cite[Lemma 5]{xxMozgovoyKac} for the completely analogous proof in the case of \mbox{absolutely} \mbox{indecomposables} instead of absolutely simples. In Theorem \ref{MAIN THEOREM} we will reformulate this lemma in terms of the plethystic exponential $\Exp$ .
\fakesubsection{The main theorem}
We are now ready to prove the existence of counting polynomials for the \mbox{refined} counting functions \eqref{def_refined_countfct}. We begin our proof with a lemma about the element \mbox{$\varepsilon\inv = \sum_{[\calM]} e_\calM [\calM]$} discussed at \eqref{equ_coeff_epsilon_inverse}.
\begin{lem}\thmabsatz
\label{lem_logint}
Let $\bbF_q$ be suitable for $\calG$ . Denote by $\int:\cplhallalg{\bbF_q[\calG]}\to \twcplmonoidalg{$q$-}{\calT(\calG)}$ the Hall algebra integral defined in \eqref{equ_Hall_algebra_integral}. We consider $\int\left(\varepsilon\inv\right)\in  \twcplmonoidalg{$q$-}{\calT(\calG)}$ as an element of $\cplmonoidalg{\calT(\calG)}$ within this lemma. This element satisfies
\[
\log \left( \int\left(\varepsilon\inv\right) \right)
=
\sum_{m\in\calT(\calG)}\sum_{\delta\vert m} \frac{1}{\delta(1-q^{ \delta})} \countfct[absim]{}{\nicefrac{m}{\delta}} \left(q^{\delta}\right) t^m
\]
\end{lem}
\begin{proof}
Using that the coefficients $e_\calM$ of $\varepsilon\inv$ are given by \eqref{equ_coeff_epsilon_inverse}, one can calculate\footnote{This computation is completely analogous to the proof of \cite[Thm. 4.2]{xxMozgovoyReineke2009}.}
\[
\sum_{[\calM]\in \iso(\bbF_q[\calG])}\frac{e_\calM}{\#\Aut[{\bbF_q[\calG]}]{\calM}} t^{\dimvec(\calM)} = \prod_{\substack{m\in\calT(\calG),\\ c\vert m}} \left( \sum_{b\geq 0} \left( \prod_{\beta=1}^b (1-q^{c\beta}) \right)\inv .t^{b.m} \right)^{\countfct[sim]{}{m,c}(q)}
\]
in $\cplmonoidalg{\calT(\calG)}$ . So we may apply \eqref{equ_Heim} to obtain
\[
\int \left(\varepsilon\inv \right) = \prod_{\substack{m\in\calT(\calG),\\ c\vert m}} \left(\ev \circ \Exp\left( \frac{1}{1-s^c}t^m \right)\right)^{\countfct[sim]{}{m,c}(q)}
\]
By applying $\log$ and using that $\log$ and $\exp$ commute with $\ev$ we deduce
\[
\log\left(\int \left(\varepsilon\inv \right)\right)
%
%
= \sum_{\substack{m\in\calT(\calG),\\ c\vert m}} {\countfct[sim]{}{m,c}(q)}\sum_{\beta\geq 1} \frac{1}{\beta(1-q^{c\beta})} t^{\beta.m}
\]
Applying the second formula in \eqref{equ_identities_Galois_und_Schur} yields
\[
\log\left(\int \left(\varepsilon\inv \right)\right) = \sum_{\substack{m\in\calT(\calG)}} \sum_{\beta\geq 1} \sum_{c\vert m} \frac{1}{c\beta(1-q^{c\beta})}\sum_{\gamma\vert c} \mu(\gamma)\countfct[absim]{}{\nicefrac{m}{c}} \left( q^{\nicefrac{c}{\gamma}}\right) t^{\beta.m}
%
%
\]
The rest of the proof is done by substituting $n:= \nicefrac{m}{c}, \delta:= \nicefrac{c}{\gamma}, a:=\beta\gamma$ and using that
\[
\sum_{\gamma\vert a} \mu(\gamma) = \begin{cases}
1 &, \quad a=1
\\
0 &,\quad a>1
\end{cases}
\] 
\end{proof}
To formulate our main result below we define the power series
\begin{equation}
\label{equ_F}
F := \calS \Bigl( \sum_{m\in\calT(\calG)} \frac{P_m^\calG}{\countgl[\vert m\vert]} t^m {\Bigr)} \in \cplmonoidalg[{\bbQ[s]_{(s-q^\alpha)}}]{\calT(\calG)}
\end{equation}
\begin{thm}\thmabsatz
\label{MAIN THEOREM}
Let $\bbF_q$ be suitable for the finitely generated virtually free group \mbox{$\calG$ .} Define the power series
\begin{equation}
\label{equ_def_Rabsim_Rss}
\countpol[absim]{}{} := (1-s) \Log\left(\calS\inv\left( F\inv \right)\right) \quad , \quad \countpol[ss]{}{} := \Exp \left(\countpol[absim]{}{}\right)
\end{equation}
for $F$ as defined in \eqref{equ_F} and denote their coefficients by $\countpol[absim]{}{m}$ and $\countpol[ss]{}{m}$ respectively. For each dimension vector $m\in\calT(\calG)$ these coefficients satisfy $\countpol[absim]{}{m} , \countpol[ss]{}{m} \in \bbZ[s]$ and
\begin{equation}
\label{equ_claim_main_thm}
\fa \alpha\geq 1: \countpol[absim]{}{m}\left( q^\alpha\right) = \countfct[absim]{}{m} \left(q^\alpha\right) \quad ,\quad \countpol[ss]{}{m} \left( q^\alpha\right) = \countfct[absim]{}{m} \left(q^\alpha\right)
\end{equation}
\end{thm}
\begin{proof}
By {(i) in Subsection \ref{subsec_geometric_methods} and \eqref{equ_counting_refined_moduli_spaces}} it is sufficient to show $\countpol[absim]{}{m} , \countpol[ss]{}{m}\in \bbQ[s]_{(s-q^\alpha)}$ for each $\alpha$ and that $\countpol[absim]{}{m} , \countpol[ss]{}{m}$ fulfill \eqref{equ_claim_main_thm}. For each $\alpha\geq 1$ we consider the continuous $\bbQ$-algebra homomorphism
\[
\calS \circ \int : \cplhallalg{\bbF_{q^\alpha}[\calG]}\to \cplmonoidalg[\bbQ]{\calT(\calG)}
\]
Using that $\Aut{\calM_x} \cong \stab{x}(\bbF_{q^\alpha})$ by \eqref{equ_stabilizer_Aut_fctor_of_pts} and $\# (\GL{d}(\bbF_{q^\alpha}).x) = \nicefrac{\countgl[\vert m\vert] (q^\alpha)}{\# \stab{x}(\bbF_{q^\alpha})}$ for $x\in\operatorname{Rep}_m(\bbF_q[\calG])(\bbF_{q^\alpha})$ and $m\in\calT(\calG)$ , we compute
\[
\int(\varepsilon) = 
%
%
\sum_{\substack{m\in\calT(\calG),\\ 
\dimvec(\calM) = m}}
\frac{1}{\#\Aut{\calM} } t^m  =
%
%
\sum_{m\in\calT(\calG)} \frac{\#\operatorname{Rep}_m(\bbF_q[\calG])(\bbF_{q^\alpha})}{\countgl[\vert m\vert] (q^\alpha)} t^m =
\ev[q^\alpha]\left( \calS\inv(F)\right)
\]
Hence, $\int\left(\varepsilon\inv\right) = \ev[q^\alpha] \left(\calS\inv \left( F\inv \right)\right)$ for $\alpha\geq 1$ . Since we have a power series \mbox{$\int\left(\varepsilon\inv\right) \in \twcplmonoidalg{$q$-}{\calT(\calG)}$} for each power $q^\alpha$ , we consider the expression $\int\left(\varepsilon\inv\right)$ as a function in $q$-powers and denote its value in $q^\alpha$ by $\int\left(\varepsilon\inv\right)\eingeschr{q^\alpha}$ .\par%
%
%
%
%
%
Now define for $m\in\calT(\calG)$ and $\alpha\geq 1$
\[
\Lambda_m := \sum_{\delta\vert m} \frac{1}{\delta(1-s^\delta)} \countpol[absim]{}{\nicefrac{m}{\delta}}\left(s^\delta \right) \in\bbQ(s)
\; , \;
\lambda_m (q^\alpha) := \sum_{\delta\vert m} \frac{1}{\delta(1-q^{\alpha \delta})} \countfct[absim]{}{\nicefrac{m}{\delta}} \left(q^{\alpha \delta}\right) \in\bbQ
\]
By definition of $\Psi$ we have $\sum_m \Lambda_m t^m = \Psi\left((1-s)\inv \countpol[absim]{}{}\right) = \log\circ\calS\inv\left(F\inv\right)$ . On the other hand we have
\[
\sum_{m\in\calT(\calG)} \lambda_m \left(q^\alpha\right) t^m = \log \left( \int\left(\varepsilon\inv\right)\eingeschr{q^\alpha} \right) = \ev[q^\alpha] \circ\log\circ\calS\inv \left( F\inv \right)
\]
by Lemma \ref{lem_logint}, where we use that $\log$ commutes with the evaluation homomorphism $\ev[q^\alpha]$ . Hence, $\Lambda_m(q^\alpha) = \lambda_m(q^\alpha)$ holds for all $m,\alpha$ . Via induction on $\operatorname{gcd}(m)$ it can now be seen that $\countpol[absim]{}{m}\in\bbQ[s]_{(s-q^\alpha)}$ and $\countpol[absim]{}{m}\left( q^\alpha\right) = \countfct[absim]{}{m} \left(q^\alpha\right)$ for all $\alpha\geq 1$ .\footnote{For the definition of $\operatorname{gcd}(m)$ see \eqref{equ_def_gcddimvec} above.}\par%
We deduce the claim for $\countpol[ss]{}{}$ from Lemma \ref{lem_Exp(absim)_ss_kleinr}. Since $\countpol[absim]{}{m}\in\bbZ[s]$ for all $m$ , we have that \mbox{$\Psi\left(\countpol[absim]{}{}\right)\in \cplmonoidalg[{\bbQ[s]}]{\calT(\calG)}$ .} {Hence, $\countpol[ss]{}{} = \exp\circ\Psi\left(\countpol[absim]{}{}\right) \in \cplmonoidalg[{\bbQ[s]}]{\calT(\calG)}$ too.} Moreover one checks immediately that $E(q^\alpha) = \ev[q^\alpha]\circ \Psi\left( \countpol[absim]{}{} \right)$ for all $\alpha$ . So Lemma \ref{lem_Exp(absim)_ss_kleinr} shows that $\ev[q^\alpha] \left( \countpol[ss]{}{} \right) = \exp ( E(q^\alpha)) = \countfct[ss]{}{}(q^\alpha)$ for all \mbox{$\alpha\geq 1$ .}
\end{proof}
Note that the counting polynomials are independent of the choice of the suitable field $\bbF_q$ , because all objects involved in \eqref{equ_F} and \eqref{equ_def_Rabsim_Rss} are. As {already stated at the end of Subsection \ref{subsec_virtually_free_groups_and_their_group_algebras}} all statements in Theorem \ref{MAIN THEOREM} hold in the more general setting of representations of the fundamental algebra \eqref{equ_fundamental_algebra_of_semisimples} of a finite graph of finite dimensional semisimple $\bbF_q$-algebras for which each of the semisimple algebras $\calA_\alpha$ and $\calA_e'$ are completely split.\par%
The proof of the following corollary is immediate from Theorem \ref{MAIN THEOREM} and \eqref{equ_identities_Galois_und_Schur}.
\begin{kor}\thmabsatz
Let $\bbF_q$ be suitable for $\calG$ . For $m\in\calT(\calG)$ , $c\geq 1$ define
\[
\countpol[sim]{}{m,c} := 
\begin{cases}
\frac{1}{c}\sum_{\gamma \vert c} \mu\left(\gamma\right) \countpol[absim]{}{\nicefrac{m}{c}}\left( s^{\nicefrac{c}{\gamma}}\right)
&, \quad \text{if } c\vert m
\\
\hfil 0 &, \quad \text{else}
\end{cases}
\]
and $\countpol[sim]{}{m} := \sum_{c\vert m} \countpol[sim]{}{m,c}$ . The polynomials $\countpol[sim]{}{m,c},\countpol[sim]{}{m}\in\bbQ[s]$ satisfy
\[
\fa \alpha\geq 1: \countpol[sim]{}{m,c}\left( q^\alpha\right) = \countfct[sim]{}{m,c}\left( q^\alpha\right)
 \quad , \quad 
\countpol[sim]{}{m}\left( q^\alpha\right) = \countfct[sim]{}{m}\left( q^\alpha\right)
\]
\end{kor}
\section{Examples}
\label{section_examples}
\subsection{Examples for section 3}
In this section we want to provide explicit examples of the objects discussed within this paper. As all of these invariants associated to a virtually free group are derived from the invariants associated to its finite subgroups, we will start with applying the Examples \ref{bsp_dimvec}(a), \ref{bsp_eulerform} and \ref{bsp_counting_repspace} to explicit finite groups. Since our invariants are independent of the choice of a suitable field $K$ , we may without loss of generality work over $K=\bbC$ .
\fakesubsubsection{Example: finite Abelian groups}
Assume $\calF$ is a finite Abelian group of order $\#\calF=a$ . Every (absolutely) simple representation of $\calF$ is of dimension 1. Hence, $\bbC[\calF]\cong \bbC^a$ , $\calT(\calF) \cong {\bbN_0}^a$ and $\mongrad: {\bbN_0}^a \to \bbN_0$ is given by $\vert m\vert = \sum_\alpha m(\alpha)$ . $\langle - ,-\rangle_\calF$ and $P_m^\calF$ are given by
\begin{equation}
\label{equ_eulerform_counting_rep_finite}
\langle m,n\rangle_\calF = \sum_{\alpha = 0}^{a-1} m(\alpha) n(\alpha) 
\quad , \quad
P_m^\calF := \nicefrac{\countgl[\vert m\vert]}{\prod_{\alpha=0}^{a-1} \countgl[m(\gamma)]}
\end{equation}
More generally: If $\calF$ is any finite group, then $\langle - ,-\rangle_\calF$ and $P_m^\calF$ are given by \eqref{equ_eulerform_counting_rep_finite} where $c$ is the number of generators of the free commutative monoid $\calT(\calF)$ .
\fakesubsubsection{Example: dihedral groups}
Now consider the dihedral group $\bbD_c = \langle s,t\mid s^2=t^2=1= (st)^c\rangle$ of order $2c$ . First consider the case $c=2a$ even: There are 4 (absolutely) simple representations of dimension 1 and $a-1$ (absolutely) simple representations of dimension 2. Hence, $\bbC [\bbD_{2a}] \cong \bbC^4 \times \bfM_2(\bbC)^{a-1}$ , $\calT(\bbD_{2a})={\bbN_0}^{a+3}$ and $\vert m\vert = \sum_{\gamma=0}^3 m(\gamma) + 2\sum_{\gamma = 4}^{a+2} m(\gamma)$ .\par%
If $c= 2a+3$ is odd, we have 2 (absolutely) simple representations of dimension 1 and $a+1$ of dimension 2. So we have $\bbC [\bbD_{2a+3}] \cong \bbC^2 \times \bfM_2(\bbC)^{a+1}$ , $\calT(\bbD_{2a+3})={\bbN_0}^{a+3}$ and $\vert m\vert = \sum_{\gamma=0}^1 m(\gamma) + 2\sum_{\gamma = 2}^{a+2} m(\gamma)$ .\par%
%
%
%
%
%
\fakesubsubsection{Example: Amalgamated free products of cyclic groups}
\label{example_dimvec_amalgamated_free_product}
We now consider the amalgamated free product $C_a *_{C_c} C_b$ . Denote the embeddings of $C_c$ by $\iota: C_c \hookrightarrow C_a$ and $\kappa: C_c\hookrightarrow C_b$ . For each (absolutely) simple representation of $C_c$ there are $\nicefrac{a}{c}$ ones of $C_c$ and $\nicefrac{b}{c}$ ones of $C_b$ which are restricted to it. Hence, by Proposition \ref{sa_dimvec_virt_free} $\calT(C_a *_{C_c} C_b)$ is given by
\[
{\bbN_0}^a \times_{{\bbN_0}^c} {\bbN_0}^b = 
\Bigl\{
(m,n) \in {\bbN_0}^a \times {\bbN_0}^b 
\mid
\fa 0\leq \gamma< c:
\sum_{\delta=0}^{\nicefrac{a}{c}-1} m(\gamma + \delta c) = \sum_{\epsilon=0}^{\nicefrac{b}{c}-1} n(\gamma + \epsilon c)
{\Bigr\}}
\]
with $\vert (m,n)\vert = \sum_\alpha m(\alpha) = \sum_\beta n(\beta)$ . By Proposition \ref{sa_Eulerform_virt_free} the Euler form is given by
\[
\langle (m,n), (u,v)\rangle_{C_a *_{C_c} C_b} = \sum_{\alpha = 0}^{a-1} m(\alpha) u(\alpha) + \sum_{\beta = 0}^{b-1} n(\beta) v(\beta) 
- \sum_{\gamma=0}^{c-1} \sum_{\delta=0}^{\nicefrac{a}{c}-1} \sum_{\epsilon=0}^{\nicefrac{b}{c}-1}  m(\gamma + \delta c) n(\gamma + \epsilon c)
\]
Note that by permuting the entries of ${\bbN_0}^a \times {\bbN_0}^b$ we obtain a monoid isomorphism \mbox{${\bbN_0}^a \times_{{\bbN_0}^c} {\bbN_0}^b \cong \left( {\bbN_0}^{\nicefrac{a}{c}} \times_{{\bbN_0}} {\bbN_0}^{\nicefrac{b}{c}}\right)^c \cong \calT(C_{\nicefrac{a}{c}} * C_{\nicefrac{b}{c}})^c$ .}
\fakesubsubsection{Example: $\PGL{2}(\bbZ)$ and $\GL{2}(\bbZ)$}
%
%
%
Our last two examples in this subsection are the groups $\PGL{2}(\bbZ)\cong \bbD_2 *_{C_2} \bbD_3$ and $\GL{2}(\bbZ)\cong \bbD_4 *_{C_2 \times C_2} \bbD_6$ . Using Proposition \ref{sa_dimvec_virt_free} and the computation of $\calT(\bbD_c)$ above, one can compute that $\calT(\PGL{2}(\bbZ))$ is isomorphic to
\[
\left\{ 
(m,n) \in {\bbN_0}^4 \times {\bbN_0}^3 \mid m(0)+m(1) = n(0)+n(2), m(2)+m(3) = n(1)+n(2)
\right\}
\]
with $\vert (m,n)\vert = \sum_\gamma m(\gamma) = n(0)+n(1)+2n(2)$ and that $\calT(\GL{2}(\bbZ))$ is isomorphic to $\{(m,n) \in {\bbN_0}^5 \times {\bbN_0}^6 \mid (*)\}$ where $(*)$ are the three relations
\begin{center}
%
$
m(0)+m(1)=n(0)+n(1)+n(5) \; , \; m(4) = n(4) \; , \; m(2)+m(3) = n(2)+n(3)+n(5) 
$
\end{center}
and with $\vert (m,n)\vert = 2m(4) + \sum_{\gamma=0}^3 m(\gamma) = 2n(4)+2n(5)+\sum_{\delta=0}^3 n(\delta)$ .
\subsection{Examples of counting polynomials}
\label{subsec_examples_counting_polynomials}
We now want to present some examples for the counting polynomials. A first trivial example are the counting polynomials of a finite group $\calF$ : Here $\dimvec: \iso(\bbC[\calF]) \to \calT(\calF)$ is bijective by Example \ref{bsp_dimvec}. Hence, $\countpol[ss]{,\calF}{m}=1$ for all $m\in\calT(\calF)$ and $\countpol[absim]{,\calF}{m}=1$ if the unique $[\calM]\in \iso(\bbC[\calF])$ of $\dimvec(\calM)=m$ is (absolutely) simple and zero otherwise.\par%
For $\calG$ an arbitrary finitely generated virtually free group given by \eqref{equ_virt_free_decomp} one first needs to compute the free commutative monoids $\calT(\calG_i)$ and $\calT(\calG_j')$ as well as the \mbox{homomorphisms} $\calT(\iota_j)$ and $\calT(\kappa_j)$ between them as we have done above for some examples, i.e. one has to classify the representation theory of these finite groups e.g. over $\bbC$ . The rest of the computation of the counting polynomials can be done by a computer, e.g. using the accompanying SageMath code. All of the examples below (and in fact many more) have been computed in this way.
\fakesubsubsection{Example: (generalized) infinite dihedral}
First consider the group $\calG_c:= C_{2c} *_{C_c} C_{2c}$ . $\calG_c$ is a finite central extension of the infinite dihedral group $\bbD_\infty = C_2 * C_2$ . As discussed \hyperref[{example_dimvec_amalgamated_free_product}]{above} its dimension vector monoid can be written as $\calT(\calG_c)\cong \left( {\bbN_0}^2 \times_{\bbN_0} {\bbN_0}^2 \right)^c$ . For the dimension vector $m=(m_0,\ldots, m_{c-1})$ we have
\begin{equation}
\label{equ_counting_polyn_extended_dihedral}
\countpol[absim]{}{m} = \begin{cases}
1 &, \quad \text{if }\vert m\vert =1
\\
s-2 &, \quad \text{if }\esex \gamma \text{ s.t. } m_\gamma = (1,1,1,1) \;\&\; m_\delta = (0,0,0,0) \; \fa \delta\neq \gamma
\\
0 &, \quad \text{else}
\end{cases}
\end{equation}
In particular all absolutely simple representations of $\calG_c$ over a suitable field occur in dimension 1 or 2. The group $\calG_c$ is among the few groups for which it is possible to determine all the polynomials $\countpol[absim]{}{m}$ explicitly. In fact, we not only count but classify all absolutely simple representations of $\calG_c$ in Subsection \ref{subsec_classification_extended_dihedral} below.
\fakesubsubsection{Example: $\PSL{2}(\bbZ)$}
We now consider $\PSL{2}(\bbZ) \cong C_2* C_3$ with $\calT(\PSL{2}(\bbZ)) \cong {\bbN_0}^{2} \times_{{\bbN_0}} {\bbN_0}^{3}$ . For $\vert m\vert \leq 4$ those $\countpol[absim]{,\PSL{2}(\bbZ)}{m}$ which are non-zero are listed below.
\begin{center}
\resizebox{12cm}{!}{
\begin{tabular}{|c|c|}
\hline
$m$ & $\countpol[absim]{,\PSL{2}(\bbZ)}{m}$ 
\\\hline
$((1,0),(1,0,0))$ & $1$ 
\\\hline
$((1,0),(0,1,0))$ & $1$ 
\\\hline
$((1,0),(0,0,1))$ & $1$ 
\\\hline
$((0,1),(1,0,0))$ & $1$ 
\\\hline
$((0,1),(0,1,0))$ & $1$ 
\\\hline
$((0,1),(0,0,1))$ & $1$ 
\\\hline
$((1,1),(1,1,0))$ & $s-2$
\\\hline
\end{tabular}
\begin{tabular}{|c|c|}
\hline
$m$ & $\countpol[absim]{,\PSL{2}(\bbZ)}{m}$
\\\hline
$((1,1),(1,0,1))$ & $s-2$
\\\hline
$((1,1),(0,1,1))$ & $s-2$
\\\hline
$((2,1),(1,1,1))$ & $s^2-3s+3$
\\\hline
$((1,2),(1,1,1))$ & $s^2-3s+3$
\\\hline
$((2,2),(2,1,1))$ & $s^3-3s^2+5s-4$
\\\hline
$((2,2),(1,2,1))$ & $s^3-3s^2+5s-4$
\\\hline
$((2,2),(1,1,2))$ & $s^3-3s^2+5s-4$
\\\hline
\end{tabular}
}
\end{center}
For $\vert m\vert\leq 5$ all non-zero $\countpol[absim]{,\PSL{2}(\bbZ)}{m}$ in a given total dimension $\vert m\vert$ coincide. However, from total dimension $\vert m\vert=6$ on this fails as the following polynomials show.
\begin{center}
\resizebox{15cm}{!}{
\begin{tabular}{|c|c||c|c|}
\hline
$m$ & $\countpol[absim]{,\PSL{2}(\bbZ)}{m}$ & $m$ & $\countpol[absim]{,\PSL{2}(\bbZ)}{m}$
\\\hline
$((4,2),(2,2,2))$ & $s^5-4s^4+6s^3-7s^2+9s-6$
& $((3,3),(1,3,2))$ & $s^5-3s^4+5s^3-7s^2+9s-6$
\\\hline
$((3,3),(3,2,1))$ & $s^5-3s^4+5s^3-7s^2+9s-6$
& $((3,3),(2,1,3))$ & $s^5-3s^4+5s^3-7s^2+9s-6$
\\\hline
$((3,3),(2,3,1))$ & $s^5-3s^4+5s^3-7s^2+9s-6$
& $((3,3),(1,2,3))$ & $s^5-3s^4+5s^3-7s^2+9s-6$
\\\hline
$((3,3),(3,1,2))$ & $s^5-3s^4+5s^3-7s^2+9s-6$
& $((4,2),(2,2,2))$ & $s^5-4s^4+6s^3-7s^2+9s-6$
\\\hline
$((3,3),(2,2,2))$ & \multicolumn{3}{c|}{$s^7+3s^6-10s^5+3s^4+14s^3-27s^2+35s-23$}
\\\hline
\end{tabular}
}
\end{center}
The examples above suggest that there are symmetries on the sets $\calT_d(\PSL{2}(\bbZ))$ along which the counting polynomials stay the same. This is indeed the case for all of the groups $C_a *_{C_c} C_b$ and we will discuss these symmetries in Section \ref{section_structural_properties} below.\par%
\fakesubsubsection{Example: $\SL{2}(\bbZ)$}
The polynomials $\countpol[absim]{,\SL{2}(\bbZ)}{m}$ are basically the same as those for $\PSL{2}(\bbZ)$ . More generally for $m=(m_0,\ldots ,m_{c-1}) \in \calT(C_{\nicefrac{a}{c}} * C_{\nicefrac{b}{c}})^c \cong \calT(C_a*_{C_c} C_b)$ we have
\[
\countpol[absim]{,C_a *_{C_c} C_b}{m} = 
\begin{cases}
\countpol[absim]{, C_{\nicefrac{c}{e}} * C_{\nicefrac{d}{e}}}{m_\gamma} &, \quad \text{if }\esex \gamma \text{ s.t. } m_\delta = 0 \; \fa \delta\neq \gamma
\\
\hfil 0 &, \quad \text{else}
\end{cases}
\]
However, the analogous statement for the counting polynomials $\countpol[ss]{,C_a*_{C_c} C_b}{m}$ is false.
\subsection{Counting polynomials of character varieties}
We now want to give \mbox{examples} for the counting polynomials $\countpol[ss]{,\calG}{d}$ . Recall that these give the E-polynomials of the character varieties $X_\calG(\GL{d}(\bbC)) = \modulirep{\bbC[\calG]}{d}$ as discussed \hyperref[fakesubsub_E-polynomials]{in Subsection \ref{subsec_geometric_methods}}.
\ifthenelse{1=1}{
\begin{center}
\resizebox{13cm}{!}{
\begin{tabular}{|c|c||c|c|}
\hline
$d$ & \multicolumn{3}{c|}{ $\countpol[ss]{,\PSL{2}(\bbZ)}{d}$ }
\\\hline
$1$ & $6$
&
$2$ & $3s+15$
\\\hline
$3$ &
\mbo\mbo\mbo\mbo\mbo\mbo\mbo\mbo\mbo\mbo\mbo
$2 s^{2} + 12 s + 26$
\mbo\mbo\mbo\mbo\mbo\mbo\mbo\mbo\mbo\mbo\mbo\mbo
&
$4$ & $3 s^{3} + 9 s^{2} + 24 s + 39$
\\\hline
$5$ & \multicolumn{3}{c|}{ $6 s^{4} + 6 s^{3} + 24 s^{2} + 36 s + 54$ }
\\\hline
$6$ & \multicolumn{3}{c|}{ $s^{7} + 3 s^{6} - 2 s^{5} + 25 s^{4} + 56 s^{2} + 41 s + 71$ }
\\\hline
$7$ & \multicolumn{3}{c|}{ $6 s^{8} + 12 s^{7} - 30 s^{6} + 54 s^{5} + 36 s^{3} + 54 s^{2} + 66 s + 90$ }
\\\hline
$8$ & \multicolumn{3}{c|}{ $3 s^{11} + 9 s^{10} + 9 s^{9} - 33 s^{8} + 66 s^{7} - 60 s^{6} + 81 s^{5} + 24 s^{4} + 33 s^{3} + 93 s^{2} + 66 s + 111$ }
\\\hline
\end{tabular}
}
\end{center}
}{
\begin{center}
\resizebox{15cm}{!}{
\begin{tabular}{|c|c|}
\hline
$d$ & $\countpol[ss]{,\PSL{2}(\bbZ)}{d}$
\\\hline
$1$ & $6$
\\\hline
$2$ & $3s+15$
\\\hline
$3$ & $2 s^{2} + 12 s + 26$
\\\hline
$4$ & $3 s^{3} + 9 s^{2} + 24 s + 39$
\\\hline
$5$ & $6 s^{4} + 6 s^{3} + 24 s^{2} + 36 s + 54$
\\\hline
$6$ & $s^{7} + 3 s^{6} - 2 s^{5} + 25 s^{4} + 56 s^{2} + 41 s + 71$
\\\hline
$7$ & $6 s^{8} + 12 s^{7} - 30 s^{6} + 54 s^{5} + 36 s^{3} + 54 s^{2} + 66 s + 90$
\\\hline
$8$ & $3 s^{11} + 9 s^{10} + 9 s^{9} - 33 s^{8} + 66 s^{7} - 60 s^{6} + 81 s^{5} + 24 s^{4} + 33 s^{3} + 93 s^{2} + 66 s + 111$
\\\hline
\end{tabular}
}
\end{center}
}
The highest $d$ for which the author has computed $\countpol[ss]{,\PSL{2}(\bbZ)}{d}$ so far is $d=12$ . $\countpol[ss]{,\PSL{2}(\bbZ)}{12}$ is given by $ s^{25} + 3 s^{24} + 18 s^{23} + 38 s^{22} + 67 s^{21} + 48 s^{20} - 49 s^{19} - 210 s^{18} - 186 s^{17} + 329 s^{16} + 738 s^{15} - 1131 s^{14} + 141 s^{13} + 264 s^{12} + 657 s^{11} - 1067 s^{10} + 542 s^{9} - 216 s^{8} + 753 s^{7} - 786 s^{6} + 508 s^{5} + 313 s^{4} - 224 s^{3} + 476 s^{2} - 143 s + 215 $ .
\ifthenelse{1=1}{
\begin{center}
\resizebox{15cm}{!}{
\begin{tabular}{|c|c||c|c|}
\hline
$d$ & \multicolumn{3}{c|}{ $\countpol[ss]{,\SL{2}(\bbZ)}{d}$ }
\\\hline
$1$ & $12$
&
$2$ & $6s+66$
\\\hline
$3$ &
\mbo\mbo\mbo\mbo\mbo\mbo\mbo\mbo\mbo\mbo\mbo\mbo\mbo\mbo\mbo\mbo\mbo
$4 s^{2} + 60 s + 232$
\mbo\mbo\mbo\mbo\mbo\mbo\mbo\mbo\mbo\mbo\mbo\mbo\mbo\mbo\mbo\mbo\mbo
&
$4$ & $6 s^{3} + 51 s^{2} + 282 s + 615$
\\\hline
$5$ & \multicolumn{3}{c|}{$12 s^{4} + 60 s^{3} + 288 s^{2} + 876 s + 1356$}
\\\hline
$6$ & \multicolumn{3}{c|}{$ 2 s^{7} + 6 s^{6} - 4 s^{5} + 144 s^{4} + 264 s^{3} + 1062 s^{2} + 2092 s + 2636 $}
\\\hline
$7$ & \multicolumn{3}{c|}{$ 12 s^{8} + 36 s^{7} - 24 s^{6} + 132 s^{5} + 624 s^{4} + 864 s^{3} + 2916 s^{2} + 4212 s + 4680 $}
\\\hline
$8$ & \multicolumn{3}{c|}{$ 6 s^{11} + 18 s^{10} + 18 s^{9} + 12 s^{8} + 324 s^{7} - 369 s^{6} + 1122 s^{5} + 1575 s^{4} + 2532 s^{3} + 6366 s^{2} + 7620 s + 7761 $}
\\\hline
\end{tabular}
}
\end{center}
}{
\begin{center}
\resizebox{15cm}{!}{
\begin{tabular}{|c|c|}
\hline
$d$ & $\countpol[ss]{,\SL{2}(\bbZ)}{d}$
\\\hline
$1$ & $12$
\\\hline
$2$ & $6s+66$
\\\hline
$3$ & $4 s^{2} + 60 s + 232$
\\\hline
$4$ & $6 s^{3} + 51 s^{2} + 282 s + 615$
\\\hline
$5$ & $12 s^{4} + 60 s^{3} + 288 s^{2} + 876 s + 1356$
\\\hline
$6$ & $ 2 s^{7} + 6 s^{6} - 4 s^{5} + 144 s^{4} + 264 s^{3} + 1062 s^{2} + 2092 s + 2636 $
\\\hline
$7$ & $ 12 s^{8} + 36 s^{7} - 24 s^{6} + 132 s^{5} + 624 s^{4} + 864 s^{3} + 2916 s^{2} + 4212 s + 4680 $
\\\hline
$8$ & $ 6 s^{11} + 18 s^{10} + 18 s^{9} + 12 s^{8} + 324 s^{7} - 369 s^{6} + 1122 s^{5} + 1575 s^{4} + 2532 s^{3} + 6366 s^{2} + 7620 s + 7761 $
\\\hline
\end{tabular}
}
\end{center}
}
\ifthenelse{1=1}{
\begin{center}
\resizebox{12cm}{!}{
\begin{tabular}{|c|c||c|c|}
\hline
$d$ & \multicolumn{3}{c|}{ $\countpol[ss]{,\GL{2}(\bbZ)}{d}$ }
\\\hline
$1$ & $4$
&
$2$ & $s + 14$
\\\hline
$3$ & 
\mbo\mbo\mbo\mbo\mbo\mbo\mbo\mbo\mbo\mbo\mbo\mbo
$ 8 s + 28 $ 
\mbo\mbo\mbo\mbo\mbo\mbo\mbo\mbo\mbo\mbo\mbo\mbo\mbo
&
$4$ & $ 3 s^{2} + 26 s + 56 $
\\\hline
$5$ & \multicolumn{3}{c|}{$ 20 s^{2} + 56 s + 88 $}
\\\hline
$6$ & \multicolumn{3}{c|}{$ s^{4} + 8 s^{3} + 59 s^{2} + 101 s + 147 $}
\\\hline
$7$ & \multicolumn{3}{c|}{$ 8 s^{4} + 36 s^{3} + 128 s^{2} + 156 s + 212 $}
\\\hline
$8$ & \multicolumn{3}{c|}{$ 2 s^{6} + 6 s^{5} + 34 s^{4} + 96 s^{3} + 223 s^{2} + 242 s + 323 $}
\\\hline
$9$ & \multicolumn{3}{c|}{$ 4 s^{7} + 16 s^{6} - 8 s^{5} + 148 s^{4} + 140 s^{3} + 400 s^{2} + 320 s + 440 $}
\\\hline
$10$ & \multicolumn{3}{c|}{$ s^{9} + 8 s^{8} + 20 s^{7} + 23 s^{6} + 35 s^{5} + 306 s^{4} + 206 s^{3} + 647 s^{2} + 435 s + 628 $}
\\\hline
\end{tabular}
}
\end{center}
}{
\begin{center}
\resizebox{15cm}{!}{
\begin{tabular}{|c|c|}
\hline
$d$ & $\countpol[ss]{,\GL{2}(\bbZ)}{d}$
\\\hline
$1$ & $4$
\\\hline
$2$ & $s + 14$
\\\hline
$3$ & $ 8 s + 28 $
\\\hline
$4$ & $ 3 s^{2} + 26 s + 56 $
\\\hline
$5$ & $ 20 s^{2} + 56 s + 88 $
\\\hline
$6$ & $ s^{4} + 8 s^{3} + 59 s^{2} + 101 s + 147 $
\\\hline
$7$ & $ 8 s^{4} + 36 s^{3} + 128 s^{2} + 156 s + 212 $
\\\hline
$8$ & $ 2 s^{6} + 6 s^{5} + 34 s^{4} + 96 s^{3} + 223 s^{2} + 242 s + 323 $
\\\hline
$9$ & $ 4 s^{7} + 16 s^{6} - 8 s^{5} + 148 s^{4} + 140 s^{3} + 400 s^{2} + 320 s + 440 $
\\\hline
$10$ & $ s^{9} + 8 s^{8} + 20 s^{7} + 23 s^{6} + 35 s^{5} + 306 s^{4} + 206 s^{3} + 647 s^{2} + 435 s + 628 $
\\\hline
\end{tabular}
}
\end{center}
}
\ifthenelse{1=1}{
\begin{center}
\resizebox{15cm}{!}{
\begin{tabular}{|c|c||c|c|}
\hline
$d$ & \multicolumn{3}{c|}{ $\countpol[ss]{,\PGL{2}(\bbZ)}{d}$ }
\\\hline
$1$ & $4$
&
$2$ & $ 14 $
\\\hline
$3$ & $ 4 s + 28 $
&
$4$ & $ s^{2} + 13 s + 55 $
\\\hline
$5$ & 
\mbo\mbo\mbo\mbo\mbo\mbo\mbo\mbo\mbo\mbo\mbo\mbo\mbo\mbo\mbo\mbo\mbo\mbo
$ 8 s^{2} + 32 s + 84 $
\mbo\mbo\mbo\mbo\mbo\mbo\mbo\mbo\mbo\mbo\mbo\mbo\mbo\mbo\mbo\mbo\mbo\mbo\mbo
&
$6$ &
$ 6 s^{3} + 18 s^{2} + 60 s + 132 $
\\\hline
$7$ & \multicolumn{3}{c|}{$ 4 s^{4} + 16 s^{3} + 44 s^{2} + 96 s + 180 $}
\\\hline
$8$ & \multicolumn{3}{c|}{$ s^{6} + 5 s^{5} + 11 s^{4} + 40 s^{3} + 64 s^{2} + 152 s + 253 $}
\\\hline
$9$ & \multicolumn{3}{c|}{$ 4 s^{7} + 12 s^{6} - 20 s^{5} + 80 s^{4} + 16 s^{3} + 156 s^{2} + 188 s + 324 $}
\\\hline
$10$ & \multicolumn{3}{c|}{$ 6 s^{8} + 22 s^{7} - 16 s^{5} + 154 s^{4} - 6 s^{3} + 256 s^{2} + 242 s + 426 $}
\\\hline
$11$ & \multicolumn{3}{c|}{$ 4 s^{10} + 20 s^{9} + 36 s^{8} - 72 s^{7} + 72 s^{6} + 56 s^{5} + 100 s^{4} + 148 s^{3} + 228 s^{2} + 372 s + 524 $}
\\\hline
$12$ & \multicolumn{3}{c|}{$ s^{13} + 4 s^{12} + 19 s^{11} + 27 s^{10} - 25 s^{9} - 15 s^{8} + 209 s^{7} - 268 s^{6} + 303 s^{5} + 178 s^{4} + 60 s^{3} + 438 s^{2} + 420 s + 659 $}
\\\hline
\end{tabular}
}
\end{center}
}{
\begin{center}
\resizebox{15cm}{!}{
\begin{tabular}{|c|c|}
\hline
$d$ & $\countpol[ss]{,\PGL{2}(\bbZ)}{d}$
\\\hline
$1$ & $ 4 $
\\\hline
$2$ & $ 14 $
\\\hline
$3$ & $ 4 s + 28 $
\\\hline
$4$ & $ s^{2} + 13 s + 55 $
\\\hline
$5$ & $ 8 s^{2} + 32 s + 84 $
\\\hline
$6$ & $ 6 s^{3} + 18 s^{2} + 60 s + 132 $
\\\hline
$7$ & $ 4 s^{4} + 16 s^{3} + 44 s^{2} + 96 s + 180 $
\\\hline
$8$ & $ s^{6} + 5 s^{5} + 11 s^{4} + 40 s^{3} + 64 s^{2} + 152 s + 253 $
\\\hline
$9$ & $ 4 s^{7} + 12 s^{6} - 20 s^{5} + 80 s^{4} + 16 s^{3} + 156 s^{2} + 188 s + 324 $
\\\hline
$10$ & $ 6 s^{8} + 22 s^{7} - 16 s^{5} + 154 s^{4} - 6 s^{3} + 256 s^{2} + 242 s + 426 $
\\\hline
$11$ & $ 4 s^{10} + 20 s^{9} + 36 s^{8} - 72 s^{7} + 72 s^{6} + 56 s^{5} + 100 s^{4} + 148 s^{3} + 228 s^{2} + 372 s + 524 $
\\\hline
$12$ & $ s^{13} + 4 s^{12} + 19 s^{11} + 27 s^{10} - 25 s^{9} - 15 s^{8} + 209 s^{7} - 268 s^{6} + 303 s^{5} + 178 s^{4} + 60 s^{3} + 438 s^{2} + 420 s + 659 $
\\\hline
\end{tabular}
}
\end{center}
}
\subsection{Classification for $\calG_{\paramExtDihed}$}
\label{subsec_classification_extended_dihedral}
%
%
%
We will now classify all absolutely simple \mbox{representations} of $\calG_{\paramExtDihed}= C_{2{\paramExtDihed}} *_{C_{\paramExtDihed}} C_{2{\paramExtDihed}}$ over a suitable ground field. This will in particular prove that the counting polynomials $\countpol[absim]{,\calG_{\paramExtDihed}}{m}$ are given by \eqref{equ_counting_polyn_extended_dihedral}. Recall that the dimension vector monoid of $\calG_c$ is given by \mbox{$\calT(\calG_{\paramExtDihed})\cong \left( {\bbN_0}^{2} \times_{{\bbN_0}} {\bbN_0}^{2}\right)^{\paramExtDihed}$ .}
\begin{sa}\thmabsatz
\label{sa_classification_G_a}
Let $K$ be a suitable field for $\calG_{\paramExtDihed}$ .\footnote{i.e. $\chara{K}$ does not divide $2{\paramExtDihed}$ and $K$ contains a primitive $2{\paramExtDihed}$-th root of unity} Denote its group of $2{\paramExtDihed}$-th roots of unity by $\mu_{2{\paramExtDihed}}(K)$ . Consider the presentation \mbox{$\calG_{\paramExtDihed} = \langle f,g\mid f^2 =g^2, f^{2{\paramExtDihed}} =1\rangle$ .} In dimension $1$ all representations $\rho:\calG_{\paramExtDihed} \to \GL{1}(K)$ are absolutely simple and \mbox{pairwise} non-isomorphic. They are given by the set \mbox{$\{(x,y)\in \mu_{2{\paramExtDihed}}(K) \mid x^2 =y^2\}$} via the bijection $\rho\mapsto (\rho(f),\rho(g))$ .\par%
All other absolutely simple representations $\rho$ of $\calG_{\paramExtDihed}$ have dimension $2$ and their isomorphism classes are in bijection with the set
\[
\{(\overline{x},y)\in K^2 \mid \overline{x}\in\nicefrac{\mu_{2{\paramExtDihed}}(K)}{\{\pm 1\}} , y\in K\setminus \{ \pm x\}\}
\]
where an explicit representative is given by
\[
(\rho(f),\rho(g)) = \left( 
\begin{pmatrix}
x & 0\\
0 & -x
\end{pmatrix}
,
\begin{pmatrix}
y & 1\\
x^2-y^2 & -y
\end{pmatrix}
\right)
\]
%
%
%
%
\end{sa}
\begin{proof}
The case of dimension 1 is elementary. For dimension $d\geq 2$ we first note that $\rho(f)$ and $\rho(g)$ are diagonalizable with eigen values in $\mu_{2{\paramExtDihed}}(K)$ , because $\chara{K}$ is suitable.\par
Now take $h:= f^2 = g^2$ . As $h\in Z(\calG_{\paramExtDihed})$ is in the center, $\rho(h)=z.\einheitsmatrix[d]$ is a scalar matrix with $z\in\mu_{\paramExtDihed}(K)$ if $\rho$ is absolutely simple. Denote the two square roots of $z$ by $\pm x$ . By construction $\rho(f)$ and $\rho(g)$ have no eigen values except for $\pm x$ . We assume without loss of generality that
\[
\rho(f) = \begin{pmatrix}
x.\einheitsmatrix[d_1] & 0\\
0 & -x.\einheitsmatrix[d_2]
\end{pmatrix}
\]
with $d_1 = d-d_2 \neq 0, d$ and consider the action of \mbox{$\GL{d_1}(K) \times \GL{d_2}(K) \cong \stab[\GL{d}(K)]{\rho(f)}$} on $\rho(g) = \left(\begin{smallmatrix} L & M\\ N & W \end{smallmatrix}\right)$ .
Since $\rho$ is simple, we know that $t_1 M t_2\inv ,
t_2 N t_1\inv \neq 0$ for all $(t_1,t_2)$ . For $d=2$ we have $d_1=d_2=1$ and may take $(t_1,t_2)=(1,M)$ to get $(t_1,t_2).\rho(g) = \left(\begin{smallmatrix} L' & 1\\ N' & W' \end{smallmatrix}\right)$ . Using $\operatorname{Tr}(\rho(g))=0$ and $\rho(g)^2 = z.\einheitsmatrix[2]$ we obtain that  $y:=L'=-W'$ and $N'=x^2-y^2$ . This proves the claim for dimension 2.\par%
Now assume $\rho$ were an absolutely simple representation of dimension $d\geq 3$ . First we note that $d_1 = d_2$ : Denote by $c_1$ and $c_2$ the multiplicities of the eigen values $\pm x$ for the matrix $\rho(g)$ . As for $\rho(f)$ we have $0< c_1,c_2<d$ . Since every simultaneous eigen vector of $\rho(f)$ and $\rho(g)$ would span a subrepresentation of $\rho$ , the multiplicities have to fulfill
\begin{equation}
\label{zyeq}
c_\gamma + d_\delta \leq d \quad \fa
1\leq \gamma,\delta \leq 2
\end{equation}
The inequalities \eqref{zyeq} yield that $d=2r$ is even and $r=c_1=c_2=d_1=d_2$ . Furthermore we may assume that $\rho(g)$ is of the form $\rho(g) = \left(\begin{smallmatrix} L & \einheitsmatrix[r]\\ N & W \end{smallmatrix}\right)$ : By standard linear algebra arguments we may find $(t_1,t_2)\in \GL{d}(K)^2$ s.t. $t_1 M t_2\inv = \left(\begin{smallmatrix} \einheitsmatrix[\operatorname{rk}(M)] & 0\\ 0 & 0 \end{smallmatrix}\right)$ and it remains to show that $\operatorname{rk}(M)=r$ . We write $L = \left(\begin{smallmatrix} L_1 & L_2\\ L_3 & L_4 \end{smallmatrix}\right)$ and $W = \left(\begin{smallmatrix} W_1 & W_2\\ W_3 & W_4 \end{smallmatrix}\right)$ as block matrices with $L_1,W_1$ square matrices of size $\operatorname{rk}(M)$ . Using $\rho(g)^2 = z.\einheitsmatrix[d]$ we see that $W_2=0$ . This means that the last $r-\operatorname{rk}(M)$ basis elements span a subrepresentation, so by simpleness of $\rho$ we have $r= \operatorname{rk}(M)$ .\par%
By again using $\rho(g)^2 = z.\einheitsmatrix[d]$ we may deduce $\rho(g) = \left(\begin{smallmatrix} L &\;\; \einheitsmatrix[r]\\ z.\einheitsmatrix[r]-L^2 &\;\; -L \end{smallmatrix}\right)$ . Now let $v\in \overline{K}^r$ be an eigen vector of $L$ considered as a matrix over $\overline{K}$ . One checks easily that $\binom{v}{0}$ and $\binom{0}{v}$ span a two dimensional subrepresentation of the base extension $\rho\otimes_K \overline{K}$ which contradicts our assumption that $\rho$ is absolutely simple.
\end{proof}
\section{Structural properties}
\label{section_structural_properties}
\noindent We now discuss some of the main structural properties of the counting polynomials: their degree and the symmetries occuring among them. We start with the degree.
\begin{sa}\thmabsatz
Let $m\in \calT(\calG)$ be an arbitrary dimension vector and $\bbF_q$ suitable for $\calG$ . The polynomial $\countpol[ss]{}{m}$ is monic of degree $\dim \modulirep{\bbF_q[\calG]}{m}$ . 
If $\countpol[absim]{}{m} \neq 0$ , then $\countpol[absim]{}{m}$ is monic too and of the same degree
\begin{equation}
\label{equ_degree_absim}
\dim \modulirep{\bbF_q[\calG]}{m} = \dim \modulirep[absim]{\bbF_q[\calG]}{m} = 1-\langle m,m\rangle_\calG
%
%
\end{equation}
\end{sa}
\begin{proof}
We first recall a well-known theorem about counting polynomials which is due to S. Lang and A. Weil: Let $X$ be a polynomial count $\bbF_q$-scheme. If $X$ is \mbox{geometrically} irreducible, then its counting polynomial is monic of degree \mbox{$\dimension{X}$ .}\footnote{see \cite[Thm. 7.7.1]{xxpoonen}} This proves the claim on $\countpol[ss]{}{m}$ , since it is a counting polynomial of $\modulirep{{\bbF_q}[\calG]}{m}$ which is geometrically irreducible, because the connected component $\operatorname{Rep}_m(\overline{\bbF_q}[\calG])$ surjects onto \mbox{$\modulirep{\overline{\bbF_q}[\calG]}{m} \cong \modulirep{{\bbF_q}[\calG]}{m} \times_{\bbF_q} \spec{\overline{\bbF_q}}$} and is irreducible by {(iii) in Subsection \ref{subsec_geometric_methods}}.\footnote{In particular we have $\countpol[ss]{}{m}\neq 0$ as $\modulirep{\bbF_q[\calG]}{m}(\bbF_q)$ is non-empty because $\operatorname{Rep}_m(\bbF_q[\calG])(\bbF_q)$ is.}\par%
The claim on $\countpol[absim]{}{m}$ is proven analogously by replacing $\operatorname{Rep}_m(\overline{\bbF_q}[\calG])$ with its open subscheme $\operatorname{Rep}_m^{\text{absim}}(\overline{\bbF_q}[\calG])$ and it remains to prove the two equations in \eqref{equ_degree_absim}: The first equation follows from $\modulirep[absim]{\bbF_q[\calG]}{m} \subseteq \modulirep{\bbF_q[\calG]}{m}$ being open and \mbox{non-empty} if $\countpol[absim]{}{m}\neq 0$ . For the second equation we note that there is an induced $\PGL{\vert m\vert,\bbF_q}$-action on representation spaces that operates freely on $\operatorname{Rep}_m^{\text{absim}}(\bbF_q[\calG])$ and that its quotient $\nicefrac{\operatorname{Rep}_m^{\text{absim}}(\bbF_q[\calG])}{\PGL{\vert m\vert, \bbF_q}}$ is isomorphic to $M^{\text{absim}}(\bbF_q[\calG],m)$ . Hence,
\[
\dim {M^{\text{absim}}(\bbF_q[\calG],m)} = \dim \operatorname{Rep}_m^{\text{absim}}(\bbF_q[\calG]) - \dim \PGL{\vert m\vert,\bbF_q}
%
\]
Moreover we have $\dim \operatorname{Rep}_m^{\text{absim}}(\bbF_q[\calG]) = \dim \operatorname{Rep}_m(\bbF_q[\calG])$ , because $\operatorname{Rep}_m(\bbF_q[\calG])$ is geometrically irreducible. So the second equation in \eqref{equ_degree_absim} is equivalent to the identity $\deg P_m^\calG = \vert m\vert^2 -\langle m,m\rangle_\calG$
which can be verified using our general formula \eqref{equ_counting_repspace_formula}.
\end{proof}
In Subsection \ref{subsec_examples_counting_polynomials} we have seen that the counting polynomials are invariant with respect to certain symmetries on the dimension vectors of some virtually free groups $\calG$ .  More specifically there is a finite group $S_\calG$ acting on $K[\calG]$ for $K$ suitable and by functoriality on each $\calT_d(\calG)$ , $d\in\bbN_0$ such that $\countpol[absim]{}{m} = \countpol[absim]{}{n}$ and $\countpol[ss]{}{m} = \countpol[ss]{}{n}$ if $m,n\in\calT_d(\calG)$ belong to the same $S_\calG$-orbit. We will now sketch the construction of this group and its action on $K[\calG]$ .\par%
Let $K$ be a field which is suitable for $\calG$ . Hence, all of the finite dimensional group algebras in \eqref{equ_decomp_groupalg} are of the form
\begin{equation*}
C \cong K^{c_1} \times \bfM_2(K)^{c_2} \times \ldots \times \bfM_e(K)^{c_e}
\end{equation*}
We construct the group $S_\calG$ and its actions iteratively and we start with the case of (group algebras of) finite groups: The symmetric group $S_{c_\epsilon}$ acts naturally on $\bfM_\beta(K)^{c_\epsilon}$ via $\tau.(M_1,\ldots , M_{c_\epsilon}) = (M_{\tau(1)},\ldots, M_{\tau(c_\epsilon)})$ for each $1\leq \gamma\leq c$ , hence, $S_\calC := S_{c_1}\times \ldots \times S_{c_e}$ acts on $\calC$ via $K$-algebra automorphisms.\par
Now assume $A,B,C$ are finite groups acting via $K$-algebra automorphisms on $K$-algebras $\calA,\calB,\calC$ and assume we are given group homomorphisms $A,B\to C$ and injective $K$-algebra homomorphisms $\calC\hookrightarrow \calA,\calB$ which are $A$- and $B$-equivariant. Then $A\times_C B$ acts naturally on $\calA *_\calC \calB$ via $K$-algebra automorphisms.\par%
Finally assume we have group homomorphisms $\varphi,\theta: A\to C$ and $K$-algebra embeddings $\iota,\kappa: \calC\hookrightarrow \calA$ such that $\iota$ is $A$-equivariant with respect to $\varphi$ and $\kappa$ via \mbox{$\theta$ .} Then $\operatorname{Eq}(\varphi,\theta)\subseteq A$ acts naturally on $\HNN{\calA}{\calC}{\iota,\kappa}$ via $K$-algebra automorphisms.\par%
All of the discussion so far works without any assumptions on the involved \mbox{algebras.} However, to iteratively get an induced action on $K[\calG]$ from the actions on the group algebras $K[\calG_i]$ and $K[\calG_j']$ we need group homomorphisms $S_{\calG_{s(j)}} \to S_{\calG_j'} \leftarrow S_{\calG_{t(j)}}$ for each $j$ such that the embeddings $K[\calG_{s(j)}] \hookleftarrow K[\calG_j'] \hookrightarrow K[\calG_{t(j)}]$ become $S_{\calG_{s(j)}}$- and $S_{\calG_{t(j)}}$-equivariant. If $\calG_j'$ is the trivial group for each $j$ , this obstruction is trivial and we obtain an action of $\prod_{i=0}^I S_{\calG_i}$ on $K[\calG]$ .\par%
However, in general this is a non-trivial combinatorial task which is why we assume from now on that $\calG_i$ is Abelian for each $0\leq i\leq I$ .\footnote{Of course one could also consider a hybrid situation where for each $j$ we have $\calG_j'$ being trivial or $\calG_{s(j)}$ and $\calG_{t(j)}$ being Abelian.} Hence, each of the $\calC$ above is of the form $K^c$ , $c=\dimension[K]{\calC}$ with an action of the symmetric group $S_{c}$ . We consider an injective $K$-algebra homomorphism $\iota: K^c \hookrightarrow K^b$ and denote by $e_\gamma'$ , $0\leq \gamma <c$ the $\gamma$-th standard basis vector of $K^c$ and by $e_\beta$ , $0\leq \beta<b$ the $\beta$-th standard basis vector of $K^b$ . Both $(e_\gamma')_\gamma$ and $(e_\beta)_\beta$ are systems of pairwise orthogonal central primitive idempotents. Hence, there is a partition
\begin{equation}
\label{equ_partition_symmetry_action}
\bbI:=\{0,1,\ldots, b-1\} = \bigsqcup_{\gamma= 0}^{c-1} \bbI_\gamma
\end{equation}
such that $\iota(e_\gamma')=\sum_{\beta\in \bbI_\gamma} e_\beta$ .\footnote{Here we use commutativity of $K^b$ $-$ in the non-commutative case $(\iota(e_\gamma'))_\gamma$ would still be pairwise orthogonal idempotents, but in general not central.} Recall that each (absolutely) simple $K^b$-module is \mbox{isomorphic} to precisely one of the modules $e_\beta.K^b$ and that the \mbox{(absolutely)} simple $K^c$-modules analogously are given by $e_\gamma'.K^c$ . By \mbox{construction} of the \mbox{partition} \eqref{equ_partition_symmetry_action} we have $\iota^*(e_{\beta}.K^b)\cong e_\gamma'.K^c$ for each $\beta\in \bbI_\gamma$ . Now consider the subgroup
\[
\overline{S_b} :=\{\tau\in S_b\mid \fa 0\leq \gamma<c: \esex \!!\; 0\leq \overline{\tau}(\gamma)<c : \tau(\bbI_\gamma) = \bbI_{\overline{\tau}(\gamma)} \}
\]
i.e. the subgroup of those $\tau\in S_b$ that preserve the partition \eqref{equ_partition_symmetry_action}. We obtain a group homomorphism $\overline{S_b}\to S_c, \tau\mapsto\overline{\tau}$ with respect to which the $K$-algebra embedding $\iota$ is $\overline{S_b}$-equivariant and applying the iterative process described above yields a finite group $S_\calG$ acting on $K[\calG]$ via $K$-algebra automorphisms.\par%
By functoriality every group action on $K[\calG]$ via $K$-algebra automorphisms yields an induced action on $M(\bbF_q[\calG],d)$ , $M^{\text{absim}}(\bbF_q[\calG],d)$ and $\calT_d(\calG)$ for each $d\in\bbN_0$ . If $m,n\in\calT_d(\calG)$ lie in the same orbit, then we obtain isomorphisms
\[
M^{\text{absim}}(\bbF_q[\calG],m)\cong M^{\text{absim}}(\bbF_q[\calG],n)\quad , \quad M(\bbF_q[\calG],m) \cong M(\bbF_q[\calG],n)
\]
So in particular the counting polynomials for $m$ and $n$ coincide.\par%
\begin{bsp}\thmabsatz
We consider the case of $\calG = C_a *_{C_c} C_b$ for $a,b\geq 2$ , $c$ a common divisor of $a,b$ . The partitions \eqref{equ_partition_symmetry_action} of $\bbI:=\{0,\ldots, a-1\}$ and $\bbJ:=\{0,1,\ldots, b-1\}$ are given by $\bbI_\gamma = \{\alpha\mid \alpha\equiv \gamma \mod(c)\}$ , $\bbJ_\gamma = \{\beta\mid \beta\equiv \gamma \mod(c)\}$ . This determines the subgroups $\overline{S_a}\subseteq S_a$ and $\overline{S_b}\subseteq S_b$ .\par%
The action of $S_{\calG} = \overline{S_a} \times_{S_c} \overline{S_b}$ on $\calT(C_a *_{C_c} C_b)\cong {\bbN_0}^a \times_{{\bbN_0}^c} {\bbN_0}^b$ coincides with the restriction of the natural $S_a\times S_b$-action on ${\bbN_0}^a\times {\bbN_0}^b$ .
\end{bsp}
\end{document}